\documentclass[francais]{article}

\usepackage{amsmath,amsfonts,amssymb,amsthm}

\usepackage[francais,english]{babel}
\usepackage[T1]{fontenc}
\usepackage[latin1]{inputenc}
\usepackage{textcomp}
\usepackage{enumerate}

\usepackage{graphicx}
\usepackage{epsfig}
\usepackage{color}

\usepackage[all]{xy}

\newcommand{\ob}{\pi}
\newcommand{\C}{\mathbb{C}}
\newcommand{\R}{\mathbb{R}}
\newcommand{\RX}{\mathbb{R}_{\tau}X}
\newcommand{\RA}{\mathbb{R}A}
\newcommand{\RC}{\mathbb{R}C}
\newcommand{\CP}{\mathbb{C}P}
\newcommand{\RP}{\mathbb{R}P}
\newcommand{\Z}{\mathbb{Z}}
\newcommand{\N}{\mathbb{N}}
\newcommand{\Mor}{\mathcal{M}or}
\newcommand{\RMor}{\mathcal{M}or^{\R}}
\newcommand{\OP}{\mathcal{O}_{\CP^1}}
\newcommand{\Nu}{\mathcal{N}_u}
\newcommand{\Nud}{\mathcal{N}_{u_0}}
\newcommand{\Nua}{\mathcal{N}_{u_1}}
\newcommand{\Nui}{\mathcal{N}_{u,-\ob_i(\underline{z})}}
\newcommand{\Nuid}{\mathcal{N}_{u_0,-\ob_i(z^0)}}

\newcommand{\Nuidtd}{\mathcal{N}_{u_0,-\ob_i(\tilde{z}^0)}}
\newcommand{\nb}{k_d}
\newcommand{\ev}{ev_{\nb}^d}
\newcommand{\Rev}{\mathbb{R}_{\tau}ev_{\nb}^d}
\newcommand{\evk}{ev_{k}^d}
\newcommand{\red}{red^d}
\newcommand{\Revk}{\mathbb{R}_{\tau}ev_{k}^d}
\newcommand{\evl}{ev_{k_d-1}^d}
\newcommand{\Reg}{Reg_d(X)}

\newcommand{\K}{\mathcal{D}^d_{d',k'}}
\newcommand{\Kd}{\mathcal{K}_{d',k'}^{d}}

\newcommand{\Kg}{\Mc \setminus \M}
\newcommand{\RK}{\mathbb{R}_{\tau}\mathcal{K}_{\nb}^d}
\newcommand{\RKd}{\mathbb{R}_{\tau}\mathcal{K}_{d',k'}^{d}}
\newcommand{\RKg}{\mathbb{R}_{\tau}\mathcal{K}_{\nb}^{d}}

\newcommand{\RKk}{\mathbb{R}_{\tau}\mathcal{K}_{k}^d}
\newcommand{\W}{\omega_{\gamma}}
\newcommand{\Bl}{\mathfrak{Bl}}

\newcommand{\M}{\mathcal{M}_{k_d}^d(X)}
\newcommand{\Mk}{\mathcal{M}_{k}^d(X)}

\newcommand{\Mless}{\overline{\mathcal{M}}_{k_d-1}^d(X)}
\newcommand{\Mc}{\overline{\mathcal{M}}_{k_d}^d(X)}

\newcommand{\Mck}{\overline{\mathcal{M}}_{k}^d(X)}
\newcommand{\Mco}{\overline{\mathcal{M}}_{0}^d(X)}

\newcommand{\TMcl}{T_{\overline{\mathcal{M}}_{k_d-1}^d(X)^*}}
\newcommand{\TMc}{T_{\overline{\mathcal{M}}_{k_d}^d(X)^*}}
\newcommand{\TMck}{T_{\overline{\mathcal{M}}_{k}^d(X)^*}}

\newcommand{\TMm}{T_{\map}\mathcal{M}_{k_d}^d(X)}

\newcommand{\TMmc}{T_{\mapg}\Mc}

\newcommand{\RM}{\mathbb{R}_{\tau}\mathcal{M}_{k_d}^d(X)}
\newcommand{\RMk}{\mathbb{R}_{\tau}\mathcal{M}_{k}^d(X)}
\newcommand{\RMl}{\mathbb{R}_{\tau}\mathcal{M}_{\nb+\ell}^d(X)}
\newcommand{\RMo}{\mathbb{R}_{\tau}\mathcal{M}_{0}^d(X)}
\newcommand{\RMco}{\mathbb{R}_{\tau}\overline{\mathcal{M}}_{0}^d(X)}
\newcommand{\RMc}{\mathbb{R}_{\tau}\overline{\mathcal{M}}_{k_d}^d(X)}
\newcommand{\RMck}{\mathbb{R}_{\tau}\overline{\mathcal{M}}_{k}^d(X)}
\newcommand{\RMcl}{\mathbb{R}_{\tau}\overline{\mathcal{M}}_{k_d+\ell}^d(X)}

\newcommand{\TRM}{T_{\mathbb{R}_{\tau}\mathcal{M}_{\nb}^d(X)^*}}
\newcommand{\TRMk}{T_{\mathbb{R}_{\tau}\mathcal{M}_{k}^d(X)^*}}
\newcommand{\TRMo}{T_{\mathbb{R}_{\tau}\mathcal{M}_{0}^d(X)^*}}
\newcommand{\TRMc}{T_{\mathbb{R}_{\tau}\overline{\mathcal{M}}_{\nb}^d(X)^*}}
\newcommand{\TRMck}{T_{\mathbb{R}_{\tau}\overline{\mathcal{M}}_{k}^d(X)^*}}
\newcommand{\TRMcl}{T_{\mathbb{R}_{\tau}\overline{\mathcal{M}}_{\nb+\ell}^d(X)^*}}
\newcommand{\TRMco}{T_{\mathbb{R}_{\tau}\overline{\mathcal{M}}_{0}^d(X)^*}}

\newcommand{\mapg}{(C,\underline{z},u)}
\newcommand{\maprg}{(C^1 \cup C^2,\underline{z},u)}
\newcommand{\mapr}{(C_{\star},\underline{z}^{\star},{u}_{\star})}
\newcommand{\map}{(C,\underline{z},u)}
\newcommand{\mapd}{(C_0,\underline{z}^0,{u}_0)}
\newcommand{\mapdtd}{(C_0,\underline{\tilde{z}}^0,{u}_0)}
\newcommand{\mapa}{(C_1,\underline{z}^1,{u}_1)}
\newcommand{\mapatd}{(C_1,\underline{\tilde{z}}^1,{u}_1)}
\newcommand{\mapt}{(C_t,\underline{z}^t,{u}_t)}
\newcommand{\mapttd}{(C_t,\underline{\tilde{z}}^t,{u}_t)}
\newcommand{\maprtd}{(C_{\star},\underline{\tilde{z}}^{\star},{u}_{\star})}

\newcommand{\TRMmc}{T_{\mapg}\RMc}%
\newcommand{\TRMm}{T_{\map}\RMc}

\newcommand{\TRMd}{T_{\gamma(0)}\RMc}
\newcommand{\TRMdtd}{T_{\tilde{\gamma}(0)}\RMc}
\newcommand{\TRMa}{T_{\gamma(1)}\RMc}
\newcommand{\TRMatd}{T_{\tilde{\gamma}(1)}\RMc}

\newcommand{\basm}{(e_1,\dots,e_{k_d},f_1,\dots,f_{\nb})}

\newcommand{\basd}{(e_1^0,\dots,e_{k_d}^0,f_1^0,\dots,f_{\nb}^0)}
\newcommand{\basdtd}{(e_1^0,\dots,e_{k_d}^0,\tilde{f}_1^0,\dots,\tilde{f}_{\nb}^0)}
\newcommand{\basa}{(e_1^1,\dots,e_{k_d}^1,f_1^1,\dots,f_{\nb}^1)}

\newcommand{\bast}{(e_1^t,\dots,e_{k_d}^t,\tilde{f}_1^t,\dots,\tilde{f}_{\nb}^t)}
\newcommand{\basttd}{(\tilde{e}_1^t,\dots,\tilde{e}_{k_d}^t,\tilde{f}_1^t,\dots,\tilde{f}_{\nb}^t)}

\DeclareMathOperator{\coker}{coker}
\DeclareMathOperator{\Ker}{Ker}
\DeclareMathOperator{\Fix}{fix}

\DeclareMathOperator{\codim}{codim}

\theoremstyle{plain}
\newtheorem{theo}{Théorème}
\newtheorem*{theo*}{Théorème}
\newtheorem{lemm}{Lemme}
\newtheorem*{coro}{Corollaire}
\newtheorem{prop}{Proposition}
\newtheorem*{prop*}{Proposition}

\theoremstyle{definition}
\newtheorem{defi}{Définition}
\newtheorem*{nota}{Notation}

\theoremstyle{remark}
\newtheorem*{rema}{Remarque}

\makeatletter
\newcommand\@makefntextsans[1]{%
    \parindent 0em%
    \noindent%
    \hb@xt@0em{\hss}%
    #1}
\def\footnotetextsans{%
     \@ifnextchar [\@xfootnotenextsans%
       {\@footnotetextsans}}
\def\@xfootnotenextsans[#1]{%
  \begingroup%
     \csname c@\@mpfn\endcsname #1\relax%
  \endgroup%
  \@footnotetextsans}
\long\def\@footnotetextsans#1{\insert\footins{%
    \reset@font\footnotesize%
    \interlinepenalty\interfootnotelinepenalty%
    \splittopskip\footnotesep%
    \splitmaxdepth \dp\strutbox \floatingpenalty \@MM%
    \hsize\columnwidth \@parboxrestore%
    \color@begingroup%
      \@makefntextsans{%
        \rule\z@\footnotesep\ignorespaces#1\@finalstrut\strutbox}
    \color@endgroup}}
\makeatother

\begin{document}

\author{Nicolas Puignau}

\title{Première classe de Stiefel-Whitney des espaces d'applications stables réelles en genre zéro vers une surface convexe}
\date{}
\selectlanguage{francais}
\maketitle

\selectlanguage{english}
\begin{abstract}
Let $(X,c_X)$ be a convex projective surface equipped with a real structure. The space of stable maps $\overline{\mathcal{M}}_{0,k}(X,d)$ carries different real structures induced by $c_X$ and any order two element $\tau$ of permutation group $S_k$ acting on marked points. Each corresponding real part $\R_{\tau}\overline{\mathcal{M}}_{0,k}(X,d)$  is a real normal projective variety. As the singular locus is of codimension bigger than two, these spaces thus carry a first Stiefel-Whitney class for which we determine a representative in the case $k=c_1(X)d-1$ where $c_1(X)$ is the first Chern class of $X$. Namely, we give a homological description of these classes in term of the real part of boundary divisors of the space of stable maps.
\end{abstract}
\footnotetextsans{\noindent Keywords: moduli spaces, rational curves, real enumerative geometry.}
\footnotetextsans{AMS Classification: 14F25, 14N35, 14P25, 53B99.}

\setcounter{tocdepth}{2}
\setcounter{secnumdepth}{4}

\selectlanguage{francais}
\section*{Introduction}

Soient une surface projective complexe $X$, une classe d'homologie $d$ dans $H_2(X,\Z)$ et deux entiers naturels $g$ et $k$. L'espace des applications stables $\overline{\mathcal{M}}_{g,k}(X,d)$ est un espace de modules pour les courbes de genre $g$ réalisant la classe $d$ dans $X$ et munies de $k$ points marqués. Ces espaces constituent le cadre de la théorie de Gromov-Witten et sont par conséquent un objet fondamental de géométrie énumérative. Dans le cas du genre zéro $\overline{\mathcal{M}}_{0,k}(X,d)$ possède des propriétés remarquables, en particulier lorsque $X$ est convexe (voir p. ex. \cite{B-M,F-P,W3}). Une variété projective $X$ est \emph{convexe} lorsqu'elle est lisse et que pour tout morphisme holomorphe $u:\CP^1 \to X$ on a $H^1(\CP^1,u^*T_X)=0$.\\

{\it Dans ce qui suit nous nous intéressons exclusivement aux courbes de genre zéro. Aussi, nous omettrons la mention $g=0$ et $\overline{\mathcal{M}}_{0,k}(X,d)$ sera noté $\overline{\mathcal{M}}_{k}^d(X)$}.\\

Supposons $X$ convexe (par exemple $\CP^2$ ou $\CP^1 \times \CP^1$) alors l'espace des applications stables $\overline{\mathcal{M}}_{k}^d(X)$ est une variété projective normale avec singularités de type orbivariété. Une orbivariété est localement le quotient d'une variété lisse sous l'action d'un groupe fini. Une application stable est une classe d'isomorphisme pour les paramétrages d'une courbe munie de points marqués ; elle est associée à la donnée $(C,z_1,\dots,z_k,u)$ où $C$ est une surface de Riemann nodale de genre zéro appelée \emph{source}, $(z_1,\dots,z_k)$ un $k$-uplet de points non singuliers de $C$ et $u : C \to X$ un morphisme holomorphe dont l'image réalise la classe $d$ (cf. \cite{F-P}). Lorsque $u$ n'est pas un revêtement multiple, on dit que l'application stable est simple. On note $\Mck^*$ le lieu des applications stables simples, c'est un ouvert dense de $\Mck$ contenu dans la partie lisse. On note $\Mk$ le lieu des applications stables dont la source est $\CP^1$, c'est un ouvert dense de $\Mck$. La dimension de $\overline{\mathcal{M}}_{k}^d(X)$ est $c_1(X)d-1+k$, où $c_1(X)$ désigne la première classe de Chern de $X$. Le morphisme d'évaluation $ev_k^d:(C,z_1,\dots,z_k,u) \in \Mck \mapsto (u(z_1),\dots,u(z_k)) \in X^k$ est un morphisme entre variétés projectives de même dimensions lorsque $k=c_1(X)d-1$. On note $\nb=c_1(X)d-1$.\\
Supposons que $X$ soit équipée d'une structure réelle (une involution antiholomorphe) $c_X:X \to X$ et que $d$ vérifie ${c_X}_*(d)=-d$. Soit $\tau$ un élément d'ordre deux du groupe des permutations à $k$ éléments ; il définit une structure réelle $c_{\tau}:(x_1,\dots,x_k) \mapsto \left(c_X(x_{\tau(1)}),\dots,c_X(x_{\tau(k)})\right)$ sur $X^k$. De même on obtient sur $\overline{\mathcal{M}}_{k}^d(X)$ une structure réelle $c_{\mathcal{M},\tau}$ induite par $(z_1,\dots,z_k,u) \in  (\CP^1)^{k}\times \Mor_d(X) \mapsto \left(conj(z_{\tau(1)}),\dots,conj(z_{\tau(k)},c_X\circ u\circ conj)\right) \in (\CP^1)^{k}\times \Mor_d(X)$ où $conj$ est la conjugaison complexe standard sur $\CP^1$ et $\Mor_d(X)$ l'ensemble des morphismes holomorphes $u:\CP^1 \to X$ réalisant à l'image la classe $d$ (voir \ref{strucreelle}). Éventuellement, $\tau$ peut être l'identité. On note respectivement $\R_{\tau}X^{k}$ et $\RMck$ les parties réelles (le lieu fixe) de $c_{\tau}$ et $c_{\mathcal{M},\tau}$ puis $\RMck^*$ la partie réelle de $\Mck^*$. Le morphisme d'évaluation se restreint en un morphisme réel $\Revk:\RMck \to \RX^{k}$ qui est surjectif lorsque $k$ est égal à $\nb$.\\ 
L'espace $\RMc$ est une variété projective réelle normale, en particulier le lieu singulier est de codimension au moins deux. Cet espace possède donc une première classe de Stiefel-Whitney $w_1(\RMc)$ et un morphisme de dualité au premier grade $H^1(\RMc,\Z/2\Z) \to H_{2\nb-1}(\RMc,\Z/2\Z)$. Dans \cite{W3} Welschinger détermine une collection de sous-variétés dans $\RMc$ aui contienent un représentant pour le dual de la première classe de Stiefel-Whitney. Plus précisément, les parties réelles des composantes irréductibles de la frontière $\Mc \setminus \M$ qui sont écrasées (de codimension au moins deux à l'image) par le morphisme d'évaluation réel contiennent un représentant dual de $w_1(\RMc^*)$. La frontière de l'espace des modules est un diviseur qui correspond au lieu des applications stables dont la source est une courbe réductible. On note $red^d=\{d',d'' \in H_2(X,\Z) | d'+d''=d\}$. Pour un élément $d' \in red^d$ et un entier $k' \leq \nb$, désignons par $\RKd$ la fermeture du lieu des applications stables dans $\RMc^*$ ayant comme source l'union de deux droites projectives $\CP^1 \cup_{\xi} \CP^1$ sécantes en un point $\xi$ et dont une des branches contient $k'$ points marqués et réalise la classe $d'$ dans $X$ (on suppose $k'>k_{d'}$ ou bien $k' \geq 2$ si $d'=0$ pour des raisons de stabilité). Enfin, pour $D$ une classe d'homologie de codimension un, on désigne par $D^{\vee}$ son dual dans $H^1(\RMck,\Z/2\Z)$. La proposition $4.5$ de \cite{W3} peut se ramener à la suivante (voir partie \ref{prelim}).

\begin{prop*}[Welschinger]
La première classe de Stiefel-Whitney de la partie réelle $\RMc^*$ de $\Mc^*$ s'écrit
$$w_1(\RMc^*)=(\Rev)^*[w_1(\RX^{\nb})]+\sum_{\substack{d' \in red^d\\k_{d'}+1<k'\leq \nb}}\epsilon_{d',k'}.[\RKd]^{\vee}$$
où $\epsilon_{d',k'} \in \{0,1\}$. On convient que $\R\mathcal{K}_{0,1}=\emptyset$.
\end{prop*}

Nous déterminons exactement quels sont les termes, dans cette somme, qui sont affectés d'un facteur non nul ; c'est-à-dire quelles sont les composantes qui contribuent effectivement à représenter un élément dual pour la première classe de Stiefel-Whitney de la sous-variété lisse $\RMc^* \subset \RMc$. Puisque le lieu singulier est de codimension au moins deux, ceci définit une première classe de Stiefel-Whitney pour $\RMc$. Le résultat obtenu est le suivant.

\begin{theo*}
Soit $X$ une surface projective convexe équipée d'une structure réelle $c_X$, $d$ une classe d'homologie dans $H_2(X,\Z)$ telle que ${c_X}_*(d) =-d$ et $\tau$ une permutation de $\nb=c_1(X)d-1$ éléments d'ordre au plus deux. La première classe de Stiefel-Whitney de la partie réelle $\RMc$ s'écrit
$$w_1(\RMc)=(\Rev)^*w_1(\RX^{\nb})+\sum_{\substack{d' \in \red\\k_{d'}<k'\leq \nb}}\epsilon_{d',k'}.[\RKd]^{\vee}$$
avec $\epsilon_{d',k'}$ appartient à $\{0,1\}$ et où $\epsilon_{d',k'}=1$ si et seulement si $k'-k_{d'}=2 \mod (4)$ ou $k'-k_{d'}=3 \mod (4)$.
\end{theo*}

La démonstration de ce théorème peut se généraliser en toutes dimensions dès que la première classe de Stiefel-Whitney admet ce type de représentant en terme de diviseurs de la frontière (en particulier lorsqu'on sait définir le degré réel du morphisme d'évaluation). Par exemple lorsque $X= \CP^3$ équipé de la conjugaison complexe et que $d$ désigne un degré supérieur à trois nous avons obtenu un résultat similaire pour $w_1(\mathbb{R}\overline{\mathcal{M}}_{2d}^d(\CP^3))$ (cf \cite{these}).\\

{\bf Remerciements.} Une partie de ce travail est contenue dans ma thèse de doctorat. L'étude de ce problème m'a été suggérée par Jean-Yves Welschinger. Je le remercie ainsi que Stepan Yu. Orevkov pour avoir encadré ce travail au cours de nombreuses discussions. Je remercie Viatcheslav Kharlamov et Jean-Claude Sikorav ainsi que le relecteur pour les multiples remarques utiles à la rédaction de ce texte. Enfin, je remercie Antonio D\'iaz-Cano Oca\~na et le département d'algèbre de l'université Complutense de Madrid pour leur grande hospitalité.

\tableofcontents

\section{Préliminaires}\label{prelim}

On appelle \emph{structure réelle} sur une variété complexe la donnée d'un automorphisme involutif antiholomorphe. Le lieu des éléments fixés par une telle application est appelé lieu fixe ou la partie réelle. On considère $(X,c_X)$ une surface convexe $X$ équipée d'une structure réelle $c_X$ dont le lieu fixe $\R X$ est non vide.

\subsection{Espace des applications stables}

Soit $d$ une classe d'homologie dans $H_2(X,\Z)$ réalisable par une courbe rationnelle et telle que ${c_X}_*(d) =-d$. On dira que $d$ est rationnelle réelle. On note $\Mor_d(X)=\{u:\CP^1 \to X\ |\ u_*[\CP^1]=d\}$ l'ensemble des morphismes holomorphes depuis la droite projective $\CP^1$ qui réalisent la classe $d$ dans $X$. Soit $k \geq 0$ un entier, on note $Diag_k=\{(z_1,\dots,z_k) \in (\CP^1)^k\ |\ \exists\ i\neq j : z_i=z_j\}$. On note $Aut(\CP^1)$ les automorphismes holomorphes de la droite projective. Le groupe de M{\oe}bius des automorphismes holomorphes et antiholomorphes de $\CP^1$ agit sur $\left((\CP^1)^k \setminus Diag_k\right) \times \Mor_d(X)$ de la façon suivante
$$\begin{array}[u]{c}\mathcal{M}oeb\times  \left((\CP^1)^k \setminus Diag_k\right) \times \Mor_d(X) \to \left((\CP^1)^k \setminus Diag_k\right) \times \Mor_d(X)\\
(\varphi;z_1,\dots,z_k,u) \longmapsto \left\{\begin{array}{l}(\varphi(z_1),\dots,\varphi(z_k),u\circ\varphi^{-1}) \textrm{ si }\varphi \in Aut(\CP^1)\\
(\varphi(z_1),\dots,\varphi(z_k),c_X\circ u\circ\varphi^{-1}) \textrm{ sinon.}\end{array}\right.\end{array}$$

L'action restreinte du sous-groupe des automorphismes holomorphes de $\CP^1$ définit un espace quotient $\Mk=\frac{\left((\CP^1)^k \setminus Diag_k\right) \times \Mor_d(X)}{Aut(\CP^1)}$ qui est une variété quasi projective. Il en existe une compactification $\Mck$ appelée espace des applications stables (cf. \cite{K-M}). Ce dernier est un espace de modules pour les courbes $k$-pointées de genre zéro réalisant la classe d'homologie $d$ dans $X$ (cf. \cite{F-P}). Un élément de $\Mck$ est une classe d'équivalence définie par la donnée $(C,z_1,\dots,z_k,u)$ d'une courbe nodale $C$ de genre zéro, de $k$ points non singuliers de $C$ tous distincts $(z_1,\dots,z_k)$ et d'un morphisme holomorphe $u : C \to X$ tel que la classe d'homologie à l'image $u_*[C]$ soit $d$. La relation d'équivalence $(C,z_1,\dots,z_k,u) \cong (D,\zeta _1,\dots,\zeta_k,v)$ est donnée par un isomorphisme $\varphi : C \stackrel{\sim}{\rightarrow} D$ qui respecte les points marqués et les applications~: $\varphi(z_i)=\zeta_i$ pour $i \in \{1,\dots,k\}$ et $u=v \circ \varphi$. De plus $(C,z_1,\dots,z_k,u)$ est \emph{stable} c'est-à-dire que son groupe d'automorphisme est fini. La stabilité est équivalente à exiger que les composantes irréductibles de la source $C$ envoyées par une application constante sur la classe nulle possèdent au moins trois points spéciaux : points marqués ou singularités (cf. \cite{F-P}). On note $\Mck^*$ le lieu des applications stables dont le groupe d'automorphisme est trivial ; c'est un ouvert dense de $\Mck$. On note $c_1(X)$ la première classe de Chern du fibré tangent de $X$. On rappelle le résultat suivant.

\begin{prop}[{\it Cf.} \cite{F-P}]\label{Mck}
Avec les notations précédentes,
\begin{enumerate}
\item $\Mck$ est une variété projective normale de pure dimension $$c_1(X)d+k-1\ ;$$
\item $\Mck$ est localement le quotient d'une variété lisse par un groupe fini ;
\item $\Mck^*$ est lisse et est un espace de module fin pour les applications stables sans automorphismes, muni d'une courbe universelle $\overline{\mathcal{U}}_{\nb}^d$.
\end{enumerate}
\end{prop}
Le second énoncé indique que $\Mck$ est une orbivariété. C'est la convexité de $X$ qui assure que $\Mor_d(X)$ est une variété lisse. On note $T_{\Mck^*}$ le fibré tangent de la variété $\Mck^*$.

\subsubsection{Structures réelles}\label{strucreelle}

L'action résiduelle du quotient ${\mathcal{M}oeb}/{Aut(\CP^1)} \cong \Z/2\Z$ définit une structure réelle $c_{\mathcal{M}}$ sur $\Mk$ qui s'étend en une structure réelle $c_{\overline{\mathcal{M}}}$ sur $\Mck$ (cf. \cite{W3}). Le lieu fixe de $c_{\overline{\mathcal{M}}}$ est un espace de module $\R\Mck$ qui compactifie un espace de paramètres pour les courbes réelles de genre zéro et de classe $d$ dans $X$  munies de $k$ points marqués dans $\R X$ (cf. \cite{these}). Une courbe est réelle lorsqu'elle est invariante sous l'action de $c_X$. Lorsque $k=0$ on obtient une structure réelle sur $\Mor_d(X)$ dont on note $\RMor_d(X)$ la partie réelle. Lorsque $k\neq 0$ le groupe des permutations de $k$ éléments $S_k$ agit sur l'indexation des $k$ points marqués par$$\begin{array}{ccc}S_k ×\Mck& \to & \Mck \\(\sigma;C,z_1,\dots,z_k,u)& \mapsto & (C,z_{\sigma(1)},\dots,z_{\sigma(k)},u).\end{array}$$ Soit $\tau$ un élément d'ordre deux de $S_k$. L'action combinée de $\tau$ et $c_{\overline{\mathcal{M}}}$ induit une structure réelle $c_{\overline{\mathcal{M}},{\tau}}$ sur $\Mck$ dont on note $\RMck$ la partie réelle. Cette dernière compactifie un espace de paramètres pour les courbes réelles de genre zéro et de classe $d$ dans $X$ avec des points marqués dans $\R X$ (d'indexations invariantes par $\tau$) et des paires de points conjugués pour $c_X$ (d'indexations permutées par $\tau$). Ces structures réelles ne se distinguent que par la classe de conjugaison de $\tau \in S_k$ (cf. \cite{W3}).
\begin{prop}[{\it Cf.} \cite{W3}]\label{RMc}Soit $\tau \in S_k$ une permutation d'ordre au plus deux,
$\RMck$ est une variété réelle projective et normale de dimension $c_1(X)d+k-1$.
\end{prop}

\subsection{Frontière}\label{diviseurs}

La frontière de $\Mck$, c'est-à-dire $\Mck \setminus \Mk$ est le lieu des application stables dont la source est réductible. C'est une union de diviseurs dans $\Mck$ (cf. \cite{F-P}). Soit $(A,B;d_A,d_B)$ un quadruplet qui vérifie~:
\begin{enumerate}[\indent 1.]
\item $A\sqcup B$ est une partition de $\{1,\dots,k\}$ ;
\item $d_A+d_B=d \in H_2(X,\Z)$ ;
\item si $d_A=0$ (resp. $d_B=0$), alors $|A|\geq 2$ (resp. $|B| \geq 2$).
\end{enumerate}
Il existe un diviseur $D(A,B;d_A,d_B)$ de $\Mck$ (cf. \cite{F-P}) qui est le lieu des applications stables $\mapg$ vérifiant~:
\begin{enumerate}[\indent a.]
\item $C=C_A\cup C_B$ est l'union de deux courbes $C_A$ et $C_B$ de genre zéro sécantes en un point nodal $\xi$ ;
\item les points marqués indexés par $A$ (resp. $B$) sont dans $C_A$ (resp. $C_B$) ;
\item $u_*[C_A] = d_A$ et $u_*[C_B]=d_B$, ce que l'on écrira $u \in \Mor_{d_A+d_B}(X)$.
\end{enumerate}
On convient que lorsque ($d_A=0$, $|A|<2$) ou ($d_B=0$, $|B|<2$), on note par extension $D(A,B;d_A,d_B)$ pour désigner l'ensemble vide (conditions de stabilité). Lorsque $X=\CP^2$ ou $X=\CP^1\times \CP^1$, les diviseurs de la frontière sont des variétés irréductibles.
\begin{figure}[htb]
\begin{center}
\input{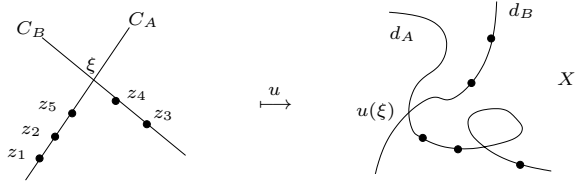}
\caption{Un élément générique du diviseur $D(\{1,2,5\},\{3,4\};d_A,d_B)$}
\end{center}
\end{figure}\\
Pour une classe rationnelle $d \in H_2(X,\Z)$, l'ensemble des éléments $\delta$, rationnels ou nuls et tels que $d-\delta$ soit rationnelle ou nulle, est noté $red^d \subset H_2(X,\Z)$. Pour un entier $k' \leq k_d$ et $d' \in red^d$, on pose \begin{equation}\Kd=\bigcup_{\substack{|A| = k' \\ d_A=d'}} D(A,B;d_A,d_B).\end{equation} Un point générique de $\Kd$ est une application stable $(C^1 \cup_{\{\xi\}} C^2,u_1 \cup u_2,z)$ telle que $C^i = \CP^1$ pour $i \in \{1,2\}$ et $\#(z \cap C^1)=k'$, $u_1^*(C^1)=d'$. Remarquons que $\R_{\tau}\mathcal{K}_{0,1}=\emptyset$ par stabilité.

\begin{prop}\label{decomp frontiere}
La frontière de $\Mc$ est la réunion $\bigcup_{\substack{d' \in red^d \\ k_{d'} < k' \leq \nb}} \Kd$.
\end{prop}

\begin{proof}
La réunion des diviseurs $\bigcup_{\substack{A\sqcup B=\{1,\dots,k\} \\ d_A+d_B=d}} D(A,B;d_A,d_B)$ est un recouvrement de la frontière (cf. \cite{F-P}), donc une simple réécriture montre que $\Mck \setminus \Mk=\bigcup_{\substack{d' \in red^d \\ 0 \leq k' \leq \nb}} \Kd$. Il suffit de remarquer que $\Kd$ est égal à $\mathcal{K}_{d-d',k-k'}^d$. Si on choisit d'imposer la condition $k'>k_{d'}$ on obtient l'unicité de l'écriture puisque des deux choix envisageables pour $A$ et $B$, un et un seul est réalisable de la sorte. En effet $\nb=c_1(X)d-1$ donc si $d'+d''=d$ alors $k_{d'}+k_{d''}=k_{d}-1$. Donc $k'>k_{d'}$ est équivalent à $\nb-k'=k''\leq \nb-k_{d'}-1 =k_{d''}$.
\end{proof}

Soit $\tau$ une permutation d'ordre deux ou l'identité dans $S_k$, on note $\RKk$ (resp. $\RKd$) le lieu de la frontière (resp. de $\Kd$) fixe pour $c_{\overline{\mathcal{M}},\tau}$. Lorsque $k=\nb$, la proposition \ref{decomp frontiere} admet le corollaire suivant.

\begin{coro}\label{decomp reelle frontiere}Pour $\tau$ une permutation d'ordre deux ou l'identité dans $S_{\nb}$, la frontière réelle de $\Mc$ est
$\RK = \bigcup_{\substack{d' \in \red \\ k_{d'} < k' \leq \nb}} \RKd$.
\end{coro}

\subsection{Morphismes}

\subsubsection{Morphisme d'évaluation}\label{eval}
Soit $\tau \in S_k$ une permutation d'ordre deux, on définit une structure réelle sur $X^{k}$ par $c_{\tau}:(x_1,\dots,x_k) \mapsto \left(c_X(x_{\tau(1)}),\dots,c_X(x_{\tau(k)})\right)$ dont on note $\RX^k$ la partie réelle.
L'application d'évaluation $ev_k^d:(C,z_1,\dots,z_k,u) \in \Mck \mapsto (u(z_1),\dots,u(z_k)) \in X^k$ est un morphisme entre variétés projectives. Lorsque $k=c_1(X)d-1$ c'est un morphisme surjectif entre variétés de même dimension. On pose $\nb=c_1(X)d-1$\label{nb}. Le morphisme d'évaluation est équivariant sous l'action de $S_k$ et de $\Z/2\Z$, c'est donc un morphisme réel $ev_k^d: (\Mck,c_{\overline{\mathcal{M}},{\tau}}) \to (X^k,c_{\tau})$ entre variétés réelles.
\begin{defi}
On appelle \emph{morphisme d'évaluation réel} la restriction du morphisme d'évaluation $$\begin{array}[c]{cccc} \Revk : & \RMck & \to & \RX^k \\ & (C,z_1,\dots,z_k,u) & \mapsto & (u(z_1),\dots,u(z_k)).\end{array}$$
\end{defi}
Lorsqu'il n'y a pas d'ambiguïté, on note $\underline{z}$ pour un $k$-uplet $(z_1,\dots,z_k) \in C^k$.
Soit $\mapg \in \Mck^*$, la différentielle de $u$ induit un morphisme de faisceaux $0 \rightarrow T_C \xrightarrow{du} u^*T_X$ dont on note $\mathcal{N}_{u}$ le faisceau quotient appelé \emph{faisceau normal} de $u$. Lorsque on a une \emph{immersion}, $\Nu$ est le faisceau d'un fibré en droites sur $\CP^1$. Par la formule d'adjonction on a l'égalité $\label{degnormal}\Nu \cong \OP(c_1(X)d-2)$.

\begin{prop}[{\it Cf.} \cite{W3}]\label{kerev} Soit $\map \in \Mk^*$ un application stable sans automorphisme ; le noyau et le conoyau de la différentielle $d\ev$ en $\map$ sont respectivement
\begin{equation}\ker(d|_{\map}\evk) \cong H^0(\CP^1,\Nu \otimes \OP(-z))\end{equation} et \begin{equation}\coker(d|_{\map}\evk) \cong H^1(\CP^1,\Nu \otimes \OP(-z)).\end{equation}
\end{prop}

Lorsque $k=\nb$, le morphisme d'évaluation réel est un difféomorphisme local au voisinage d'un point régulier. On note $\Reg$ le lieu régulier du morphisme d'évaluation $\ev:\Mc \to X^{\nb}$, c'est un ouvert dense de $\Mc$.

\subsubsection{Morphismes d'oubli}\label{oubli}
Une composante irréductible d'une courbe de genre zéro est appelée \emph{branche}. Soient $l \leq k$ des entiers, il existe un morphisme entre variétés projectives de $\Mck$ vers $\overline{\mathcal{M}}_l^d(X)$ qui \og oublie \fg\ certains points marqués de la source de $\mapg \in \Mck$ puis contracte les branches devenues instables par insuffisance de points marqués (cf. \cite{F-P}).
\begin{defi}
On appelle \emph{morphisme d'oubli} et on note $\ob$ l'application
\begin{equation}\begin{array}{cccc}\ob :&\Mck &\to& \Mco \\& \mapg &\mapsto& (C_{\ob}^{stab},u).\end{array}\end{equation}
Pour $i \in \{1,\dots,k\}$ on note $\ob_i$ le morphisme qui oublie le $i$-ième point marqué $$\begin{array}{cccc}\ob_i :&\Mck &\to& \overline{\mathcal{M}}_{k-1}^d(X) \\& (C,z_1,\dots,z_k,u) &\mapsto& (C_{\ob_i}^{stab},z_1,\dots,\hat{z}_i,\dots,z_k,u).\end{array}$$
où $C_{\ob}^{stab}$ désigne la source éventuellement contractée $C \twoheadrightarrow C_{\ob}^{stab}$ pour stabiliser.
\end{defi}
Lorsque cela est utile (p. ex. en \ref{chemintd}) on spécifie l'espace de départ $\Mck$ d'un morphisme d'oubli par la notation $\ob^k$ ; de même, pour une partie $\{i_0,\dots,i_n\}$ de $\{1,\dots,k\}$, on note $\ob^k_{i_0,\dots,i_n}$ pour l'application composée $\ob^{k-n}_{i_1}\circ\dots\circ\ob^k_{i_0}$ qui consiste à oublier la liste $\{i_0,\dots,i_n\}$ de points marqués.

\begin{lemm}\label{propkero}
Soit $\mapg \in \Mck^*$ tel que $\underline{z}$ soient des points réguliers de $C_{\ob}^{stab}$ ; alors on a une décomposition $\ker d|_{\mapg}\ob = \bigoplus_{i=1}^k \ker d|_{\mapg}\ob_i$ et un isomorphisme $\ker d|_{\mapg}\ob \cong T_{\underline{z}}C^k$.
\end{lemm}

\begin{proof}
Pour tous $i,j \in \{1,\dots,\nb\}$ distincts, les morphismes $\ob_i$ et $\ob_j$ commutent donc $\ker d|_{\mapg}\ob= \bigoplus_{i=1}^{k} \ker d|_{\mapg}\ob_i$. Soit $i \in \{1,\dots,\nb\}$, en l'absence d'automorphisme le morphisme $\ob_i$ décrit la courbe universelle au-dessus de l'image $\ob_i(\Mck^*)$ (cf. \cite{B-M}). Le groupe des automorphismes d'une application stable étant un sous-groupe du groupe des automorphismes de son image par un morphisme d'oubli, $\overline{\mathcal{M}}_{k-1}^d(X)^*$ est inclus dans $\ob_i(\Mck^*)$. On en déduit un isomorphisme entre $(\ob_i)^{-1}(C_{\ob_i}^{stab},z_1,\dots,\hat{z}_i,\dots,z_k,u)$ et $C_{\ob_i}^{stab}$. Par hypothèse $z_i$ est un point régulier de la fibre $C_{\ob_i}^{stab}$, or les branches ne se contractent par $\ob_i$ que sur des points singuliers, donc $C_{\ob_i}^{stab}=C$. Finalement au point $(C,z_1,\dots,z_k,u) \cong z_i$, on a $\ker d_{\mapg}\ob_i = T_{z_i}C_{\ob_i}^{stab}=T_{z_i}C$. On en déduit un isomorphisme entre $\ker d|_{\mapg}\ob$ et $T_{\underline{z}}C^k$ lorsque $\underline{z}$ sont des points réguliers de $C_{\ob}^{stab}$.
\end{proof}
On note $\Ker d\ob$ le sous-faisceau du fibré tangent de $\TMck$ défini par le noyau de la différentielle d'un morphisme d'oubli. C'est un sous-fibré de rang un restreint au lieu des applications stables simples non contractées (en particulier sur $\Mc^*$).

\section{Résultats}\label{chap2}

\subsection{Invariants de Welschinger}\label{invw}

Soit $A$ une courbe réelle de $X$ avec pour singularités éventuelles des points doubles ordinaires, on dit que $A$ est \emph{nodale}. Les points doubles de $A$ peuvent être complexes ou réels. On distingue deux types de points doubles réels : ceux qui sont à l'intersection de deux tangentes réelles et ceux qui sont à l'intersection de deux tangentes complexes conjuguées. Le premier est un point double réel \emph{non isolé}, le second est un point double réel \emph{isolé}.
On définit la \emph{masse} de $A$, notée $m(A)$ comme le nombre de points doubles réels \emph{isolés} de $A$.

\begin{figure}[htb]
\begin{center}
\input{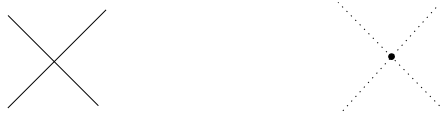}
\caption{Points doubles ordinaires réels et leurs tangentes}
\end{center}
\end{figure}
\begin{theo}[Welschinger \cite{W1}]\label{thwelschi}
Soient $d \in H_2(X,\Z)$ telle que ${c_X}_*(d)=-d$ et $\mathcal{P}$ une collection de $c_1(X)d-1$ points de $X$ invariante pour $c_X$ et dont exactement $r$ éléments sont réels. On note $\mathcal{A_P}$ pour l'ensemble des courbes réelles de genre zéro qui contiennent $\mathcal{P}$. Alors, pour $\mathcal{P}$ générique, $\mathcal{A_P}$ est fini et ne contient que des courbes rationnelles nodales. De plus, la somme \begin{equation}\label{defWd}\sum_{A \in \mathcal{A_P}}(-1)^{m(A)}\end{equation} ne dépend pas de $\mathcal{P}$.
\end{theo}
On note ce nombre $\mathcal{W}_{d,r}(X)$, c'est un invariant de Welschinger de $(X,c_X)$. Remarquons que $\vert \mathcal{W}_{d,r}(X) \vert$ est une borne inférieure pour le nombre de courbes rationnelles réelles de $(X,c_X)$ réalisant la classe $d$ et qui contiennent un collection générique de $r$ points réels et $\frac{\nb-r}{2}$ paires de points complexes conjugués.

\subsection{Aspects topologiques}

Soit $K \subset \Mc$ la réunion des diviseurs de la frontière dont l'image par le morphisme d'évaluation est de codimension au moins deux. On dit d'une composante irréductible de $K$ qu'elle est \emph{écrasée} par le morphisme d'évaluation. On note $\R_{\tau}K$ le lieu de $K$ fixe pour $c_{\overline{\mathcal{M}},\tau}$. Munissons $\RX^{\nb}$ du système de coefficients entiers tordus $\mathcal{Z}$ et considérons la classe fondamentale qui lui est associée dans l'homologie singulière  $[\RX^{\nb}] \in H_{2\nb}\left( \RX^{\nb},\mathcal{Z}\right)$ (voir \cite{Steenrod}). On note $\mathcal{Z}^*$ le système de coefficients locaux sur $\RMc$ relevé de $\mathcal{Z}$ par le morphisme d'évaluation réel.

\begin{prop}[Welschinger \cite{W3}]\label{classfonda}
Soit $\mapg \in \RMc$ un point régulier du morphisme d'évaluation et $\R_{\tau}\mathcal{M}$ la composante connexe de $\RMc$ qui le contient. Il existe une unique classe fondamentale $[\R_{\tau}\mathcal{M}]$ à bord dans $\R_{\tau}K \cap \R_{\tau}\mathcal{M}$ tel que l'homomorphisme induit par le morphisme d'évaluation réel $H_{2\nb}(\R_{\tau}\mathcal{M},\R_{\tau}\mathcal{M}\setminus\{\mapg\};\mathcal{Z}^*) \to H_{2\nb}\left( \RX^{\nb},\RX^{\nb}\setminus \{u(\underline{z})\};\mathcal{Z}\right)$ envoie $[\R_{\tau}\mathcal{M}]$ sur $(-1)^{m(A)}.[\RX^{\nb}]$ où $A$ est la courbe définie par $u(C)$.
\end{prop}

Autrement dit, il existe une classe fondamentale $[\RMc]$ à bord dans $\R_{\tau}K$ et telle que le morphisme d'évaluation réel $\Rev$ induise un homomorphisme entre $ H_{2\nb}(\RMc,\R_{\tau}K;\mathcal{Z}^*)$ et $H_{2\nb}( \RX^{\nb},\mathcal{Z})$ qui envoie la classe $[\RMc]$ sur la classe $\mathcal{W}_{d,\tau}(X).[\RX^{\nb}]$. Puisque $\RMc$ est une variété normale, son lieu singulier est de codimension au moins deux. Elle admet donc une première classe de Stiefel-Whitney et un morphisme de dualité au premier grade $H_{2\nb-1}(\RMc,\Z/2\Z)\to H^1(\RMc,\Z/2\Z) $. Pour une classe d'homologie de codimension un $[\R_{\tau}D] \in H_{2\nb-1}(\RMc,\Z/2\Z)$, on note $[\R_{\tau}D]^{\vee} \in H^1(\RMc,\Z/2\Z)$ son dual. Le théorème \ref{thwelschi} admet le corollaire suivant.

\begin{coro}[Welschinger \cite{W3}]\label{W3}
La première classe de Stiefel-Whitney de la variété projective normale $\RMc$ s'écrit : $$w_1(\RMc)=(\Rev)^*w_1(\RX^{\nb})+\sum_{D \subset K} \epsilon(D).[\R_{\tau}D]^{\vee}$$ où $\epsilon(D) \in \{0,1\}$ et la somme est prise sur l'ensemble des composantes irréductibles $D$ de la frontière de $\Mc$ écrasées par le morphisme d'évaluation et dont la partie réelle $\R_{\tau}D$ possède au moins une composante connexe de codimension un dans $\RMc$.
\end{coro}

Les points critiques du morphisme d'évaluation sont contenus dans le lieu des courbes non immergées (proposition \ref{kerev}) et dans la frontière de $\Mc$. Les composantes irréductibles de la frontière (voir \ref{diviseurs}) ne sont pas toutes écrasés par le morphisme d'évaluation. Plus précisément, on a la proposition suivante.

\begin{prop}\label{imcodim1}
L'image d'un diviseur de la frontière $D(A,B ; d_A,d_B)$ par le morphisme d'évaluation $\ev : \Mc \to X^{\nb}$ est de codimension un dans $X^{\nb}$ si et seulement si $|A|=k_{d_A}$ ou $|A|=k_{d_A}+1$.
\end{prop}

\begin{proof}
Soit $D(A,B ; d_A,d_B)$ un diviseur de la frontière de $\Mc$ tel que défini en \ref{diviseurs}. Pour $\mapg \in \Mc$ on note $\underline{z}_A=(z_{i_1},\dots,z_{i_{|A|}})$ le $|A|$-uplet tel que $A=\{i_1,\dots,i_{|A|}\}$. Soit $\overline{\mathcal{M}}_{|A|+1}^{d_A}(X) \times_X \overline{\mathcal{M}}_{|B|+1}^{d_B}(X)$ le produit fibré relativement aux morphismes d'évaluation $e_A$ et $e_B$ associés au point marqué supplémentaire $e_{A,B} : (C^{A,B},\underline{z}_{A,B},\xi,u) \in \overline{\mathcal{M}}_{|A,B|+1}^{d_{A,B}}(X) \mapsto u(\xi) \in X$. Il existe un isomorphisme $D(A,B ; d_A,d_B) \cong \overline{\mathcal{M}}_{|A|+1}^{d_A}(X) \times_X \overline{\mathcal{M}}_{|B|+1}^{d_B}(X)$ décrit dans \cite{F-P}. Le morphisme d'évaluation se factorise $\mapg \in D(A,B ; d_A,d_B) \mapsto  \left( (C_A,\underline{z}_{A},\xi,u_A),(C_B,\underline{z}_B,\xi,u_B)\right) \in \overline{\mathcal{M}}_{|A|+1}^{d_A}(X) ×\overline{\mathcal{M}}_{|B|+1}^{d_B}(X) \mapsto u(\underline{z}) \in X^{\nb}$. Supposons $|A|> k_{d_A}+1$, ce qui entra\^ine $|B|\leq k_{d_B}-2$, alors $\dim \overline{\mathcal {M}}_{|B|+1}^{d_B}(X) \leq 2k_{d_B}-1$. Puisque $\dim X^{k_{d_B}}=2k_{d_B}$, l'image du diviseur par le morphisme composé est de codimension au moins égale à deux dans $X^{\nb}=X^{k_{d_A}+k_{d_B}+1}$. La symétrie dans le rôle de $A$ et de $B$ achève le raisonnement.
\end{proof}

Dans \cite{W1}, Welschinger montre que les diviseurs de la frontière qui sont de codimension un à l'image est un lieu régulier pour le morphisme d'évaluation. Une conséquence de la proposition \ref{imcodim1} est que seuls les diviseurs qui vérifient $|A|>k_{d_A}+1$ ou $|B|>k_{d_B}+1$ peuvent contribuer à un représentant dual de la première classe de Stiefel-Whitney de la partie réelle. Plus précisément et avec les notations introduites en \ref{diviseurs} on a la proposition suivante.

\begin{prop}\label{indept}
La première classe de Stiefel-Whitney de $\RMc$ s'écrit~:
$$w_1(\RMc)=(\Rev)^*w_1(\RX^{\nb})+\sum_{\substack{d' \in \red\\k_{d'}+1<k'\leq \nb}}\epsilon_{d',k'}.[\RKd]^{\vee}$$
où $\epsilon_{d',k'} \in \{0,1\}$ dépend seulement des paramètres $d'\in \red$ et $k' \in \N$.
\end{prop}

\begin{proof}
D'après le corollaire de la proposition \ref{decomp frontiere}, la frontière de $\Mc$ est la réunion $\bigcup_{\substack{d' \in red^d \\ k_{d'} < k' \leq \nb}} \Kd$ et d'après la proposition \ref{imcodim1}, $K$ est complémentaire à la réunion $\bigcup_{\substack{d' \in \red}} \mathcal{K}_{d',k_{d'}+1}^d$. Autrement dit, les composantes connexes de $\RK$ de codimension un dans $\RMc$ et qui sont écrasées par le morphisme d'évaluation réel sont contenues dans la réunion $\bigcup_{\substack{d' \in \red \\ k_{d'}+1 < k' \leq \nb}} \RKd$. Pour conclure, on remarque que l'indexation sur les points marqués ne joue aucun rôle dans le calcul de la première classe de Stiefel-Whitney. En effet, l'application induite par une permutation $\sigma \in S_{\nb}$ des points marqués est un automorphisme de $\Mc$ (cf. \cite{K-M}) qui permute les diviseurs de la frontière $D(A,B,d_A,d_B)$ de même calibre $(|A|,|B|,d_A,d_B)$. Pour $\tau \in S_{\nb}$ une permutation d'ordre deux, cet automorphisme induit un isomorphisme réel entre $(\Mc,c_{\overline{\mathcal{M}},{\tau}})$ et $(\Mc,c_{\overline{\mathcal{M}},{\sigma\circ\tau\circ{\sigma}^{-1}}})$ puisque $\sigma\circ\tau\circ{\sigma}^{-1}\circ\sigma\circ\tau=\sigma$. Or les structures réelles $c_{\overline{\mathcal{M}},{\tau}}$ sur $\Mc$ ne se distinguent que par la classe de conjugaison de $\tau$ (voir \ref{strucreelle}), c'est-à-dire que $c_{\overline{\mathcal{M}},{\tau}}=c_{{\overline{\mathcal{M}},{\sigma\circ\tau\circ{\sigma}^{-1}}}}$. Donc les propriétés homologiques de la partie réelle des diviseurs de même calibre sont équivalentes. Une composante irréductible de la frontière est contenue dans un des diviseurs de la frontière et donc dans un certain $\Kd$. Le fait qu'une composante connexe de sa partie réelle $\mathcal{D} \subset \RMc$ contribue ou pas à un représentant dual de la première classe de Stiefel-Whitney dépend uniquement des données $k'$ et $d'$. En particulier cela ne dépend pas de la composante connexe de $\RKd$ choisie.
\end{proof}

\subsection{Exposé des résultats}

Nous déterminons quelles sont les diviseurs de la frontière dont la partie réelle $\RKd$ participe effectivement à une classe duale de la première classe de Stiefel-Whitney de l'espace des modules $w_1(\RMc)$. C'est-à-dire que nous déterminons la valeur de chaque $\epsilon_{d',k'}$ dans la description précédente.

\begin{theo}\label{theo}
Soit $X$ une surface projective convexe équipée d'un structure réelle $c_X$ et $d$ une classe non nulle de $H_2(X,\Z)$ telle que $c_X(d)=-d$. On note $\nb$ le nombre $c_1(X)d-1$ et $\tau$ une permutation d'ordre au plus deux dans le groupe des permutations $S_{\nb}$. La première classe de Stiefel-Whitney de la partie réelle $\RMc$ est représentée par
$$w_1(\RMc)=(\Rev)^*w_1(\RX^{\nb})+\sum_{\substack{d' \in \red\\k_{d'}<k'\leq \nb}}\epsilon_{d',k'}.[\RKd]^{\vee}$$
avec $\epsilon_{d',k'} \in \{0,1\}$ tel que $\epsilon_{d',k'}=1$ si et seulement si $k'-k_{d'}=2 \mod (4)$ ou  $k'-k_{d'}=3\mod (4)$.
\end{theo}

\section{Démonstration}\label{demo1}

Nous cherchons à calculer le bord de la chaîne $[\RMc]$ dans le groupe des cycles $C_{2\nb}(\RMc^*,\mathcal{Z}^*)$. Fixons $\R_{\tau}\mathcal{M}$ une composante connexe de $\RMc^* \setminus \R_{\tau}K$ et choisissons $[\R_{\tau}\mathcal{M}] \in H_{2\nb}(\R_{\tau}\mathcal{M},\R_{\tau}K \cap \R_{\tau}\mathcal{M} ;\mathcal{Z}^*)$ la classe fondamentale définie par la proposition \ref{classfonda}. On considère pour une composante connexe $\mathcal{D}$ de $\R_{\tau}K$, de codimension un dans $\RMc^*$, un voisinage contractile $B$ d'un point générique de $\mathcal{D}$ de sorte que $\R_{\tau}\mathcal{M} \cap B$ ait deux composantes connexes $\R_{\tau}\mathcal{M}_0^B$ et $\R_{\tau}\mathcal{M}_1^B$ séparées par $\R_{\tau}K \cap B$. On obtient deux générateurs $[\R_{\tau}\mathcal{M}_0^B]$ et $[\R_{\tau}\mathcal{M}_1^B]$ de $H_{2\nb}(B,\partial B \cup \R_{\tau}K;\mathcal{Z}^*)$. Chacun induit un générateur $[B_0]$ (resp. $[B_1]$) de $H_{2\nb}(B,\partial B ; \mathcal{Z}^*)$. Il s'agit d'évaluer si l'application $H_{2\nb}(B,\partial B ; \mathcal{Z}^*) \to H_{2\nb}(B,\partial B \cup \R_{\tau}K;\mathcal{Z}^*)$ réalise $\pm [B] \mapsto [B_0] + [B_1]$ ou bien $\pm [B] \mapsto [B_0] - [B_1]$ en fonction du choix de $\mathcal{D}$. Pour cela, on choisit un chemin dans $B$ transverse à $\R_{\tau}K \cap B$ et qui relie deux points de part et d'autre de $\R_{\tau}K$. Le long de ce chemin on construit une trivialisation de $\TRMc$ (voir \ref{trivial}) afin d'évaluer l'image de $[B]$ dans $H_{2\nb}(B,\partial B \cup \R_{\tau}K;\mathcal{Z}^*)$ en comparant les orientations induites par $[B_0]$ et $[B_1]$.

\subsection{\'Etude du cas $\tau=id$}\label{chemin}

\renewcommand{\RX}{\mathbb{R}X}
\renewcommand{\Rev}{\mathbb{R}ev_{\nb}^d}
\renewcommand{\Revk}{\mathbb{R}ev_{k}^d}

\renewcommand{\RM}{\mathbb{R}\mathcal{M}_{k_d}^d(X)}
\renewcommand{\RMk}{\mathbb{R}\mathcal{M}_{k}^d(X)}
\renewcommand{\RMl}{\mathbb{R}\mathcal{M}_{\nb+\ell}^d(X)}
\renewcommand{\RMo}{\mathbb{R}\mathcal{M}_{0}^d(X)}
\renewcommand{\RMco}{\mathbb{R}\overline{\mathcal{M}}_{0}^d(X)}
\renewcommand{\RMc}{\mathbb{R}\overline{\mathcal{M}}_{k_d}^d(X)}
\renewcommand{\RMck}{\mathbb{R}\overline{\mathcal{M}}_{k}^d(X)}
\renewcommand{\RMcl}{\mathbb{R}\overline{\mathcal{M}}_{k_d+\ell}^d(X)}

\renewcommand{\TRM}{T_{\mathbb{R}\mathcal{M}_{\nb}^d(X)^*}}
\renewcommand{\TRMk}{T_{\mathbb{R}\mathcal{M}_{k}^d(X)^*}}
\renewcommand{\TRMo}{T_{\mathbb{R}\mathcal{M}_{0}^d(X)^*}}
\renewcommand{\TRMc}{T_{\mathbb{R}\overline{\mathcal{M}}_{\nb}^d(X)^*}}
\renewcommand{\TRMck}{T_{\mathbb{R}\overline{\mathcal{M}}_{k}^d(X)^*}}
\renewcommand{\TRMcl}{T_{\mathbb{R}\overline{\mathcal{M}}_{\nb+\ell}^d(X)^*}}
\renewcommand{\TRMco}{T_{\mathbb{R}\overline{\mathcal{M}}_{0}^d(X)^*}}

\renewcommand{\RKd}{\mathbb{R}\mathcal{K}_{d',k'}^{d}}
\renewcommand{\RK}{\mathbb{R}\mathcal{K}_{\nb}^d}
\renewcommand{\RKg}{\mathbb{R}\mathcal{K}_{\nb}^{d}}
\renewcommand{\RKk}{\mathbb{R}\mathcal{K}_{k}^d}

On fixe la structure réelle définie par l'identité de $S_{\nb}$ ce qui impose à l'image de chaque point marqué d'appartenir à la partie réelle $\RX$ (voir \ref{strucreelle}). Soit $d' \in red^d \subset H_2(X,\Z)$, $k'$ un entier strictement supérieur à $k_{d'}+1$ et $\K$ une composante connexe de $\RKd$ de codimension un dans $\RMc$. On distingue deux situations selon que $d'$ soit nulle ou pas.

\subsubsection{Chemins transverses. Cas $d'\neq0$}

On souhaite décrire un chemin transverse à $\RKd$ en un point générique bien choisit de $\K$. Une permutation de l'indexation des points marqués $\underline{z}^{\star}$ revient éventuellement à définir une autre composante connexe $(\mathcal{D}_{d',k'}^d)'$ de $\RKd$ ce qui est sans conséquence dans ce qui suit (cf. proposition \ref{indept}).

\paragraph{Choix d'un point générique.} On rappelle qu'un élément $\maprg \in \Kg$ est générique lorsque $C^i \cong \CP^1$ pour $i \in \{1,2\}$, $\xi=\{C^1 \cap C^2\}$ est un point double ordinaire et $u$ n'a que des singularités nodales et est lisse aux points marqués $\underline{z} \in (C^1\cup C^2)^{\nb}$. On note $d^{i}{'}=u_{*}[C^i]$ et $k^{i}{'}=\#(z\cap C^i)$ de sorte que $d'+d''$ soit égal à $d$ et $k'+k''$ égal à $k_d$. Un élément générique de $\K$ a donc deux branches à la source dont l'une a \og trop\fg\ de points marqués et l'autre trop peu (au sens ou le \og bon\fg\ nombre de points marqués serait $k_{d^{i}{'}}$ ou $k_{d^{i}{'}}+1$ pour $i \in\{1,2\}$). On détermine un point $\maprg$ de $\K$ tel que le morphisme $u$ envoie un certain nombre de points marqués estimés en \og trop\fg\ dans un voisinage contractile du point double $u(\xi)$ afin de travailler dans l'homologie à coefficients entiers.

\begin{defi}\label{ell}
Pour $\maprg \in \K$, on définit le nombre $\ell \in \mathbb{N}$ (on rappelle que $k'>k_{d'}+1$) par
\begin{equation}
\begin{array}{c}
\ell=k'-k_{d'}\\ \textrm{  ou  }\ell=k'-k_{d'}-1
\end{array}
\end{equation} de sorte que $\ell \equiv 0 \mod (2)$.
\end{defi}
\begin{rema}
Un sous-ensemble non vide de points réguliers dans la partie réelle d'une courbe rationnelle réelle de $X$ est dans la composante connexe isomorphe à $\RP^1$. Il possède donc un ordre cyclique défini par l'ordre cyclique sur $\RP^1$.
\end{rema}

\begin{lemm}\label{mapr}
Quitte à changer l'indexation sur les points marqués, il existe un point générique $\mapr$ de $\K$ qui, en notant $C_{\star}=C_{\star}^1\cup_{\{\xi^{\star}\}} C^2_{\star}$, vérifie $\{z^{\star}_1,\dots,z^{\star}_{k'}\} \subset C_{\star}^1$ et $\{z^{\star}_{k'+1},\dots,z^{\star}_{\nb}\} \subset C_{\star}^2$~; les points spéciaux se trouvent à l'image dans l'ordre cyclique $u_{\star}(z^{\star}_1)<\dots<u_{\star}(z^{\star}_{\frac{\ell}{2}})<u_{\star}(\xi^{\star})<u_{\star}(z^{\star}_{\frac{\ell}{2}+1})<\dots<u_{\star}(z^{\star}_{k'})$ et  $u_{\star}(\xi^{\star})<u_{\star}(z^{\star}_{k'+1})<\dots<u_{\star}(z^{\star}_{\nb})$ ; de plus $u_{\star}([z^{\star}_1,z^{\star}_{\ell}])$ est inclus dans un voisinage contractile de $u_{\star}(\xi^{\star})$. (Voir figure \ref{config}.)
\end{lemm}

\begin{figure}[htp]
\begin{center}
\input{config.pstex_t}
\end{center}
\caption{$X=\CP^2$, $d=3.[L]$, $\nb=8$, $\ell=2$ }\label{config}
\end{figure}

\begin{proof}
On pose $u_{\star}  \in \ob(\K) \subset \RMco$ un point générique dans l'image de $\K$ par le morphisme d'oubli. En particulier $(u_{\star} : C_{\star} \to X) \in \RMor_{d'+d''}(X)$ de sorte que $C_{\star}=C_{\star}^1\cup_{\xi^{\star}}C_{\star}^2$ avec $C_{\star}^i=\CP^1$ pour $i \in \{1,2\}$ et on note $A_i=u_{\star}|_{C_{\star}^i}(C_{\star}^i)$, chaque composante irréductible de la courbe $A= u_{\star}(C_{\star})$. Soit $D_{\star}$ un disque ouvert de $u_{\star}(\xi^{\star}) \in \RX$ qui définit des coordonnées locales $(x,y) \in \R^2$ sur $\RX$ centrées en $u_{\star}(\xi^{\star})$ et telles que $\R A_1 \cap D_{\star}=\{(x,y) : x=0\}$ et $\R A_2 \cap D_{\star}=\{(x,y) : y=0\}$. On pose $U_1^+=u_{\star}^{-1}(\{y>0\}_{D_{\star}} \cap \R A_1)$ et $U_1^-=u_{\star}^{-1}(\{y<0\}_{D_{\star}} \cap \R A_1)$ des ouverts de $\RC_{\star}^1 := u_{\star}^{-1}(\R A_1)$ et $U_1=U_1^+\cup U_1^-$. On choisit ainsi un $\nb$-uplet $z^{\star}=(z^{\star}_1,\dots,z^{\star}_{\nb})$ dans l'ensemble ordonné : $$(U_1^+)^{\frac{\ell}{2}} ×(U_1^-)^{\frac{\ell}{2}} ×(\RC_{\star}^1\setminus U_1)^{k'-\ell} ×(\RC_{\star}^2)^{k''} \setminus  Diag_{\nb}.$$
On définit $\mapr \in \RKg$ en considérant la classe de $(z^{\star},u_{\star}) \in (\CP^1)^{k'} ×(\CP^1)^{k''} \times \Mor_{d'+d''}(X)$ dans $\Mc$. Par définition, $\mapr$ vérifie les propriétés du lemme \ref{mapr} pour le voisinage $D_{\star}$ et quitte à changer l'indexation des points marqués il appartient à la composante $\K$.
\end{proof}

\paragraph{Choix d'un chemin.}\label{choixchemin}
On choisit un chemin qui soit transverse à $\K$ au point $\mapr$ et suffisamment \og petit\fg\ pour considérer une orientation locale. Rappelons que $\RX$ est une surface réelle convexe (en particulier elle est projective et lisse). Il existe une première classe de Stiefel-Whitney $w_1(\RX) \in H^1(X,\Z/2\Z)$ dans l'homologie singulière et on note $w_1^{\vee}(\RX) \in H_1(\RX,\Z/2\Z)$ son dual de Poincaré. On fixe un représentant de $w_1^{\vee}(\RX)$ dans le groupe des $1$-chaînes $C_1(\RX,\Z)$ et on considère son support $\W \subset \RX$. On note $\Omega_{\gamma}=\bigcup_{i=1}^{\nb} X×\dots ×\W ×\dots ×X$. Puisque $\overline{D}_{\star}$ est contractile, on peut choisir $\W$ de sorte que $\Omega_{\gamma}$ n'intersecte pas $\Rev\mapr$ et $\W$ n'intersecte pas $\overline{D}_{\star}$. Considérons un voisinage contractile (une boule) $B_{\star}$ de $\mapr$ dans $\RMc$ tel que  $B_{\star} \cap \Kg=B_{\star} \cap \K$ de sorte que $\K$ découpe $B_{\star}$ en deux composantes connexes dans $\RM$. Alors, l'ouvert $O_{\gamma}=(\Rev)^{-1}(^{\complement}\Omega_{\gamma}) \subset \RMc$ contient $\mapr$. On fixe un chemin $\gamma:[0,1] \to \RMc^*$ dans $B_{\star}\cap O_{\gamma} \subset \RMc^*$ transverse à $\K$ au point $\mapr$. Par simplicité on identifiera $\gamma$ et son image $\gamma([0,1]) \subset \RMc$. On notera aussi $\gamma_*$ l'image de $\gamma$ par le morphisme d'évaluation $\gamma_*=\Rev(\gamma) \subset \RX^{\nb}$.
Le chemin $\gamma$ est contenu dans $B_{\star}$, transverse à $\K$ en $\mapr$ et $\gamma_*$ n'intersecte pas $\Omega_{\gamma}$. On pose $\mapt=\gamma(t)$ pour $t \in [0,1]$. On distingue plus particulièrement $\mapd$ et $\mapa$ le point de \og départ\fg\ et le point d'\og arrivée\fg\ du chemin, chacun se situant dans une composante connexe différente de $\RM \cap B_{\star}$.

\paragraph{Construction d'un chemin de comparaison.}\label{chemintd}
On construit un autre chemin $\tilde{\gamma}$ associé à $\gamma$ mais à valeurs dans le lieu régulier $\Reg$. Pour cela, on associe à $\mapr$ un point générique $\maprtd$ de la frontière tel que les courbes soient identiques mais les points marqués diffèrent afin d'obtenir un point régulier du morphisme d'évaluation.

\begin{lemm}\label{maprtilde}
Soit $\mapr \in \K$ défini par le lemme \ref{mapr}. Il existe une application stable $\maprtd \in \Kg$ qui vérifie pour $0 < i \leq \ell$, $\tilde{z}^{\star}_i \in C^2$ et pour $\ell<i\leq\nb$, $\tilde{z}^{\star}_i=z^{\star}_i$~; à l'image de $C^2$ on a l'ordre cyclique $u_{\star}(\xi^{\star})<u_{\star}(\tilde{z}^{\star}_1)<\dots<u_{\star}(\tilde{z}^{\star}_{\frac{\ell}{2}})<u_{\star}(\tilde{z}^{\star}_{k'+1})<\dots<u_{\star}(\tilde{z}^{\star}_{\nb})<u_{\star}(\tilde{z}^{\star}_{\frac{\ell}{2}+1})<\dots<u_{\star}(\tilde{z}^{\star}_l)$ et $u_{\star}([\tilde{z}^{\star}_1,\tilde{z}^{\star}_{\ell}])$ est inclus dans le voisinage de $u_{\star}(\xi^{\star})$ défini par le lemme \ref{mapr}. (Voir figure \ref{configtd}.)
\end{lemm}

\begin{figure}[htp]
\begin{center}
\input{configtd.pstex_t}
\end{center}
\caption{$X=\CP^2$, $d=3.[L]$, $\nb=8$, $\ell=2$}\label{configtd}
\end{figure}

\begin{proof}
On reprend les termes de la démonstration du lemme \ref{mapr} et on considère, dans le système de coordonnées locales sur $D_{\star}$, les ouverts $U_2^+=u_{\star}^{-1}(\{x>0\} \cap \R A_2)$, $U_2^-=u_{\star}^{-1}(\{x<0\} \cap \R A_2)$ et $U_2=U_2^+\cup U_1^-$ de $\RC^2$. On choisit cette fois $\tilde{z}^{\star}=(\tilde{z}^{\star}_1,\dots,\tilde{z}^{\star}_{\nb})$ dans l'ensemble ordonné : $$(U_2^+)^{\frac{\ell}{2}} ×(U_2^-)^{\frac{\ell}{2}} ×\{z^{\star}_{\ell+1}\} ×\dots ×\{z^{\star}_{\nb}\} \setminus  Diag_{\nb}.$$
On définit ainsi $\maprtd \in \K$ en considérant la classe de $(\tilde{z}^{\star},u_{\star}) \in (\CP^1)^{k'-\ell} ×(\CP^1)^{k''+\ell} ×\Mor_{d'+d''}(X)$ dans $\Mc$.
\end{proof}

\begin{prop}\label{reg}
L'application stable $\maprtd \in \K$ du lemme \ref{maprtilde} est un point régulier du morphisme d'évaluation.
\end{prop}

\begin{proof}
Il suffit d'après la proposition \ref{imcodim1} de vérifier que $\maprtd \in \mathcal{K}_{d',k_{d'}}^d \cup \mathcal{K}_{d',k_{d'}+1}^d$. Or, par définition, $(u_{\star})_*[C_{\star}]=d'$ et $C_{\star}^1 \cap \underline{\tilde{z}}^{\star}=\{\tilde{z}^{\star}_{\ell+1},\dots,\tilde{z}^{\star}_{k'}\}$ donc $\#(C_{\star}^1 \cap \tilde{z}^{\star}) \in \{k_{d'},k_{d'}+1\}$ en fonction du choix de $\ell \in \{k'-k_{d'},k'-k_{d'}-1\}$ (voir définition \ref{ell}).
\end{proof}

On construit un chemin $\tilde{\gamma}$ transverse à la frontière $\Kg$ au point $\maprtd$ et qui soit compatible avec $\gamma$ dans le sens où les courbes et les points marqués communs coïncident le long des chemins. Pour cela, on se place dans $\RMcl^*$ et on définit $\underline{z}^{\star}\cup\underline{\tilde{z}}^{\star}:=(z^{\star}_1,\dots,z^{\star}_{\ell},\tilde{z}^{\star}_1,\dots,\tilde{z}^{\star}_{\ell},z^{\star}_{\ell+1},\dots,z^{\star}_{\nb}) \in (C_{\star})^{\nb+\ell}\setminus Diag_{\nb+\ell}$. On considère $(C_{\star},\underline{z}_{\star}\cup\underline{\tilde{z}}_{\star},u_{\star}) \in \mathcal{K}_{\nb+\ell}^d$ en prenant la classe de $(\underline{z}_{\star}\cup\underline{\tilde{z}}_{\star},u_{\star}) \in (\CP^1)^{\nb+\ell}\setminus Diag_{\nb+\ell} ×\RMor_{d'+d''}(X)$ dans $\RMcl^*$. Par définition on a  $\ob^{\nb+\ell}_{\ell+1,\dots,2\ell}(C_{\star},\underline{z}^{\star}\cup\underline{\tilde{z}}^{\star},u_{\star})=(C_{\star},\underline{z}^{\star},u_{\star})$. Quitte à restreindre $\gamma$ on peut relever dans $\RMl^* \cup \mathcal{K}_{\nb+\ell}^d$ ce chemin par le morphisme d'oubli $\ob^{\nb+\ell}_{\ell+1,\dots,2\ell}$ au point $(C_{\star},\underline{z}^{\star}\cup\underline{\tilde{z}}^{\star},u_{\star}) \in \mathcal{K}_{\nb+\ell}^d$. Étant donnée la suite exacte
\begin{displaymath}
\xymatrix{0 \ar[r] & \Ker d\ob^{\nb+\ell}_{\ell+1,\dots,2\ell} \ar[r] & \TRMcl \ar[r]_-{d\ob^{\nb+\ell}_{\ell+1,\dots,2\ell}} & \TRMc \ar[r] & 0}
\end{displaymath}
on peut définir un chemin dans $\RMcl^*$ dont la différentielle s'annule dans le sous-fibré $\Ker d\ob_{\ell+1,\dots,2\ell}^{\nb+\ell}$ de $\TRMcl$. On pose $\Gamma$ un tel chemin passant par $(\CP^1,\underline{z}^{\star}\cup\underline{\tilde{z}}^{\star},u_{\star})$ et qui relève $\gamma$ de sorte que l'image par le morphisme d'oubli $\ob_{\ell+1,\dots,2\ell}^{\nb+\ell}$ soit constante.
\begin{displaymath}
\xymatrix{ \RMcl^*\ar[r] \ar[d]_-{\ob^{\nb+\ell}_{\ell+1,\dots,2\ell}} &  \RMc^* \\ \R\overline{\mathcal{M}}_{\nb-\ell}^d(X)^* & \ar[lu]_-{\Gamma} \ar[u]_-{\gamma} \ar[l]^-{\textrm{cte}} [0,1]}
\end{displaymath}
On définit le chemin $\tilde{\gamma}:[0,1] \to \RMc^*$ transverse à $\Kg$ au point $\maprtd$ comme l'image de $\Gamma$ par le morphisme d'oubli des $\ell$ premiers points marqués
\begin{displaymath}
\xymatrix{\Gamma \subset \RMcl^* \ar[r]^{\ob^{\nb+\ell}_{1,\dots,\ell}} & \tilde{\gamma} \subset \RMc^*.}
\end{displaymath}
On confond comme précédemment le chemin et son image, tout comme on note $\tilde{\gamma}_*$ pour l'image $\Rev(\tilde{\gamma}) \subset \RX^{\nb}$. Le chemin $\tilde{\gamma}$ est inclus dans $\Reg$, l'image dans $\mathbb{R}\overline{\mathcal{M}}_{0,k_d-\ell}^d(X)^*$ de $\tilde{\gamma}$ par $\ob_{1,\dots,\ell}$ est égale à l'image de $\gamma$ par $\ob_{1,\dots,\ell}$, enfin $\tilde{\gamma}_*$ n'intersecte pas $\Omega_{\gamma}$. Comme précédemment on note $\mapttd=\tilde{\gamma}(t)$, pour $t \in [0,1]$.

\subsubsection{Chemins transverses. Cas $d'=0$}\label{degal0}
On considère une composante connexe $\mathcal{D}_{0,k'}^d$ de $\R\mathcal{K}_{0,k'}^d$ de codimension un dans $\RMc$. (On rappelle que nécessairement $k' \geq 2$ par stabilité.) On choisit un point générique $\mapr$ dans $B_{\star} \cap \K$ ; c'est-à-dire une application stable qui est une immersion et dont la source est composée de deux branches $C_{\star}=\CP^1 \cup_{\xi} \CP^1$, l'une envoyée par une application birationnelle $u_{\star}$ sur une courbe rationnelle réelle de $X$ dans la classe $d$ et l'autre envoyée sur la classe nulle au point $u_{\star}(\xi)$. Comme précédemment, quitte à changer l'indexation, on choisit que $\{z^{\star}_1,\dots,z^{\star}_{k'}\}$ appartienne à la branche de classe nulle et un ordre cyclique sur la branche de classe non nulle $u_{\star}(z^{\star}_{k'+1})<\dots<u_{\star}(z^{\star}_{\nb})$.
On définit ensuite un chemin $\gamma : t \in [0,1] \mapsto (C_t,z^t,u_t) \in \RMc^*$ transverse à $\K$ en $\mapr$ au paramètre $t_{\star}$ mais dans le \og sens\fg\ des applications. C'est-à-dire que seuls les points marqués dépendent du paramètre $t$, l'application $u_t$ étant fixée à l'image: $u_0(\CP^1)=u_{t}(\CP^1)=u_{\star}(C_{\star})=u_1(\CP^1),\ \forall t \neq t_{\star} \in [0,1]$. On exige que $\gamma(0)=\gamma(1)$ après renversement dans l'ordre cyclique : $u_0(z^0_1)=u_1(z^1_1),\dots,u_0(z^0_{k'})=u_1(z^1_{k'})$. Autrement dit, $\mapa=(\CP^1,z^0_{k'},\dots,z^0_1,z^0_{k'+1},\dots,z^0_{\nb},u_0)$. De plus, en reprenant les notations de \ref{choixchemin}, on choisit $\gamma$ suffisamment \og petit\fg\ pour que son image $\gamma_*$ n'intersecte pas $\Omega_{\gamma}$ pour le choix de $\W$ tel que $\W \cap D_{\star} = \emptyset$, où $D_{\star}$ est un voisinage contractile de $u_{\star}(\xi)$. Il n'est pas utile de définir un autre chemin $\tilde{\gamma}$.
\begin{figure}[htp]
\begin{center}
\input{gamma0.pstex_t}
\end{center}
\caption{$X=\CP^2$, $d=2.[L]$, $\nb=5$, $k'=3$}
\end{figure}

\subsubsection{Bases de l'espace tangent}\label{secdecomptan}

L'objectif de ce paragraphe est de se munir d'une base positive, relativement à une orientation locale sur $T_{\RX^{\nb}}$, de l'espace tangent $\TRMd$ au point de départ du chemin $\gamma$. On se donne différentes décompositions du fibré tangent de $\RMck^*$ pour lesquelles on définit une terminologie ad hoc.

\begin{prop}\label{decomptan}
Soit $\mapg \in \Mck^*$, on a la suite exacte
\begin{equation}\label{decomp}
\xymatrix{0 \ar[r] & T_{\underline{z}}C^{k} \ar[r] & T_{\mapg}\Mck \ar[r] & H^0(C,\Nu) \ar[r] & 0.}
\end{equation}
\end{prop}

\begin{proof}
On a la suite exacte sur le morphisme d'oubli : $0 \to T_{\underline{z}}C^{k} \to T_{\mapg}\Mck \stackrel{d\ob}{\rightarrow} T_{\mapg}\Mco \to 0$ d'après le lemme \ref{propkero}. D'autre part, l'espace tangent $T_{(C,u)}\Mco$ est l'espace des déformations à l'ordre un de la courbe immergée par $u$, c'est-à-dire $H^0(C,\Nu)$ (cf. Chap. II.1 de \cite{Ko}).
\end{proof}

\begin{defi}\label{stand}
Soit  $\mapg \in \Mck^*$ défini par une immersion. On appelle \emph{base standard} une base de $T_{\mapg}\Mck$ adaptée à la suite exacte (\ref{decomp}). Autrement dit, la donnée d'un $(k+\nb)$-uplet $(e_1,\dots,e_{k},f_1,\dots,f_{\nb})$ composé de $k$ éléments $\{e_i\}_{i=1,\dots,k}$ générateurs de chaque $T_{z_i}C$ et $\nb$ sections linéairement indépendantes $\{f_i\}_{i=1,\dots,\nb}$, dans $H^0(C,\Nu)$.
\end{defi}

Lorsque $k=\nb$ on peut affiner la définition au points réguliers du morphisme d'évaluation . En effet, si on note ${T_X}_{(i)}=\{\vec{0}\} ×\dots ×T_X ×\dots ×\{\vec{0}\}$ le sous-fibré de $T_{X^{\nb}}$ associé à la $i$-ième composante, la décomposition de $T_{X^{\nb}}=\bigoplus_{i=1}^{\nb}{T_X}_{(i)}$ se relève par le morphisme d'évaluation en tout point régulier $\mapg \in \Reg$ \begin{equation}\label{decomp3}{\TMmc}=\bigoplus_{i=1}^{\nb}(\ev)^{*}{T_{u(z_i)}X_{(i)}}.\end{equation}

\begin{nota}
Pour $\mapg \in\Reg$ et $i\in\{1,\dots,\nb\}$, on note $\Nui$ le faisceau $\Nu \otimes \mathcal{O}_{C}(-\hat{\underline{z}}^i)$ où $\hat{\underline{z}}^i$ est l'élément de $C^{\nb-1}$ obtenu en ôtant le $i$-ième point marqué dans $\underline{z}=(z_1,\dots,z_i,\dots,z_{\nb}) \in C^{\nb}$. Les sections de ce faisceau correspondent aux déformations de la courbe qui n'affectent pas la position des points marqués sauf le $i$-ième.
\end{nota}

\begin{prop}\label{decompfib}
Soit $\mapg \in \M^* \cap \Reg$, on a la suite exacte
\begin{equation}\label{eqdecompfib}
\xymatrix{0 \ar[r] & T_{z_i}C \ar[r] & (\ev)^*T_{u(z_i)}X_{(i)} \ar[r] & H^0(C,\Nui) \ar[r] & 0.}
\end{equation}
\end{prop}

\begin{proof}
Pour $i\in\{1,\dots,\nb\}$, notons $\hat{e}_i$ le morphisme composé $$\hat{e}_i:\Mc \xrightarrow{\ob_i} \Mless \xrightarrow{\evl} X^{\nb-1}$$ de sorte que $(\ev)^*(T_{X_{(i)}})=\Ker d\hat{e}_i$. On déduit du diagramme de suites exactes
\begin{displaymath}
\xymatrix{& & & 0 \ar[d]\\
& & & \Ker d\evl \ar[d]\\
0\ar[r] &\Ker d\ob_{i} \ar[r] & \TMc \ar[r]^-{d\ob_i} \ar[rd]_-{d\hat{e}_{i}}& \TMcl \ar[d]^-{d\evl} \ar[r] & 0 \\
& & & T_{X^{\nb-1}} \ar[d]\\
& & & 0}
\end{displaymath}
la suite exacte $\xymatrix{0 \ar[r] & \Ker d\ob_i \ar[r] & \Ker d\hat{e}_{i} \ar[r]^-{d\ob_i|} & d\ob_i(\Ker d\hat{e}_{i}) \ar[r] & 0}$, ainsi que l'égalité $d\ob_i(\Ker d\hat{e}_{i})=\Ker d\evl$. On conclut en appliquant les égalités $\ker d\ob_i = T_{z_i}C$ et $\ker d\evl = H^0(C,\Nui)$ de la proposition \ref{kerev}.
\end{proof}

\begin{defi}\label{model}
Soit $\map \in \M^* \cap \Reg$, une base standard $\mathcal{B}=\basm$ de $\TMm$  est une \emph{base modèle} lorsqu'elle est compatible avec la décomposition \eqref{decomp3} et la suite exacte \eqref{eqdecompfib}. Autrement dit, si tout sous-espaces vectoriels de la décomposition \eqref{decomp3} admet comme générateurs un couple $(e_h,f_l)$ de la base $\mathcal{B}$. Si de plus, $e_i \in T_{z_i}C$ et $f_i \in H^0(C,\Nui)$ pour chaque $i \in \{1,\dots,\nb\}$ on dit que la base modèle est \emph{ordonnée}.
\end{defi}

\paragraph{Construction d'une base modèle positive}\label{basd}
On s'intéresse aux bases de l'espace vectoriel réel $\TRMc$ et à leurs orientations. On se place en un point régulier $\mapg \in \Reg \cap \RMc$ tel que $\Rev\map$ n'intersecte pas $\Omega_{\gamma}$ défini en \ref{choixchemin}. Fixons une orientation $\mathfrak{o}_{\gamma}$ sur $\RX \setminus \W$. On peut relever cette orientation sur $\TRMmc$ par la fonction d'évaluation et définir une orientation $\mathfrak{o}^*=(\Rev)^*(\mathfrak{o}\times\mathfrak{o}\times \dots \times \mathfrak{o})$. Une base de $\TRMmc$ est dite positive relativement à $\mathfrak{o}_{\gamma}$ si l'orientation qu'elle induit est $\mathfrak{o}^*$. On donne une procédure pour se munir d'une base modèle ordonnée positive relativement à $\mathfrak{o}_{\gamma}$ sur $\RX \setminus \W$. Une telle base induit une orientation directe sur $\RX^{\nb} \setminus \Omega_{\gamma}$ mais aussi sur chaque espace de la décomposition  $$T_{u(\underline{z})}\RX^{\nb}=\bigoplus^{\nb}_{i=1}T_{u(z_i)}\RX_{(i)}.$$

\begin{defi}
Soient $\map \in \Reg$ et $\basm$ une base modèle ordonnée de $\TRMm$. Pour un couple de vecteurs $(e_i,f_i)$ on définit son \emph{sens} relativement à $\mathfrak{o}_{\gamma}$ $$\angle(e_i,f_i) \in \{-1,+1\}$$ par la convention suivante (au point $\map$)~:
\begin{itemize}
\item$\angle(e_i,f_i)=+1$ si $d{\Rev}\left( (e_i,f_i)\right) $ est une base positive de $T_{u(z_i)}\RX_{(i)}$~;
\item$\angle(e_i,f_i)=-1$ si $d{\Rev}\left( (e_i,f_i)\right) $ est une base négative de $T_{u(z_i)}\RX_{(i)}$.
\end{itemize}
\end{defi}

On note $A_0=u_0(C_0)$ la courbe réelle définie par le point de départ du chemin $\gamma$. Fixons une orientation $\mathfrak{o}_{\R A_0}$ sur $\R A_0$. En combinant ces deux orientations on construit une base de $\TRMd$ comme suit. On choisit $k_d$ éléments $\{e_i^0 \in T_{z^0_i}C_0;\ i=1,\dots,k_d\}$ tels que les vecteurs $\{d_{z^0_i}u_0(e_i^0)\}_{i=1,\dots,k_d}$ soient dans $T\RA_0$ et positifs relativement à $\mathfrak{o}_{\R A_0}$ puis $\nb$ éléments $\{f_i^0 ; i=1,\dots,\nb\}$ dans $H^0(C_0,\Nud)$ de sorte que $\basd$ soit une base modèle ordonnée pour $\TRMd$ avec $\angle(e^0_i,f^0_i)=+1,\ \forall i \in \{1,\dots,\nb\}$. La base $\mathcal{B}_0=\basd$ ainsi définie est positive relativement à $\mathfrak{o}_{\gamma}$. Lorsque $\mapg \in \RMc^*$, le faisceau normal $\Nu$ est un faisceau réel pour la structure induite par $c_X$ et on note $H^0_{c_X}(C,\Nu)$ la partie réelle de $H^0(C,\Nu)$. Les éléments $\{f_i^0 ; i=1,\dots,\nb\}$ de  $\mathcal{B}_0$ sont dans $H^0_{c_X}(C,\Nud)$ et chaque $f_i$ appartient à $H^0_{c_X}(C,\Nuid)$ d'après la construction.

\subsubsection{Homotopies et trivialisation. Cas $d'\neq 0$}

\paragraph{Homotopie au point.}
L'objet de ce paragraphe est de construire une nouvelle base standard pour $\TRMd$ qui induise la même orientation que $\mathcal{B}_0$ mais qui va s'intégrer plus facilement dans une trivialisation de $T|_{\gamma}\RMc$ compatible avec le morphisme d'évaluation. On renvoie à la définition de $\tilde{\gamma}$ puis on choisit une base modèle ordonnée positive de $\TRMdtd$  relativement à $\mathfrak{o}_{\gamma}$ et $\mathfrak{o}_{\R A_0}$ suivant la méthode définie en \ref{basd} qu'on note $$\tilde{\mathcal{B}}_0=(\tilde{e}_1^0,\dots,\tilde{e}_{k_d}^0,\tilde{f}_1^0,\dots,\tilde{f}_{\nb}^0).$$
Chaque élément $\tilde{f}_i$ appartient à $H^0_{c_X}(\CP^1,\Nuidtd)$ pour $i \in \{1,\dots,\nb\}$ et le $\nb$-uplet $(\tilde{f}_1^0,\dots,\tilde{f}_{\nb}^0)$ est une base de $H^0_{c_X}(\CP^1,\Nud)$.

\begin{lemm}\label{turnlemma}
L'ensemble ordonné de vecteurs $\basdtd$ définit une base standard positive de $\TRMd$.
\end{lemm}

\begin{proof}\label{turn}
On se place au point $\Gamma(0)=(\CP^1,z^0\cup\tilde{z}^0,u_0) \in \RMl^*$ défini au paragraphe \ref{chemintd}. Quitte à \og inverser\fg\ $\gamma$ (et donc $\Gamma$), on peut supposer d'après les lemmes \ref{mapr} et \ref{maprtilde} que l'ordre cyclique des points $u_0(z^0\cup\tilde{z}^0)$ sur $\RA_0$ est donné par
\begin{eqnarray*}
\begin{aligned}u_0(z^0_1)< \dots & <u_0(z^0_{\frac{\ell}{2}})<u_0(\tilde{z}^0_1)<\cdots< u_0(\tilde{z}^0_{\frac{\ell}{2}})<u_0(z^0_{k'+1})<\cdots \\ & \cdots <u_0(z^0_{\nb})<u_0(\tilde{z}^0_{\frac{\ell}{2}+1})<\dots<u_0(\tilde{z}^0_{{\ell}})<u_0(z^0_{\ell+1})<\dots<u_0(z^0_{k'})
\end{aligned}\end{eqnarray*}On convient que lorsque $\ell=k'$, la partie à droite de $u_0(\tilde{z}^0_{\frac{\ell}{2}})$ n'est pas retenue. (Voir figure \ref{config0}.)
\begin{figure}[htp]
\begin{center}
\input{config0.pstex_t}
\end{center}
\caption{$X=\CP^2$, $d=3.[L]$, $\nb=8$, $\ell=2$}\label{config0}
\end{figure}\\
Fixons des représentants $(\underline{z}^0,u_0)$, $(\underline{\tilde{z}}^0,u_0)$ dans $ (\RP^1)^{\nb} × \RMor_d(X,c_X)$ pour les points $\mapd$ et $\mapdtd$ de sorte que $\underline{z}^0_i=\underline{\tilde{z}}^0_i$ lorsque $i \in \{\ell+1,\dots,\nb\}$. D'après l'ordre cyclique, on peut construire un chemin dans $(\RP^1)^{\nb} \setminus Diag_{\nb}$ qui lie $\underline{z}^0$ à $\underline{\tilde{z}}^0$ et qui reste constant en $z^0_{\ell+1}=\tilde{z}^0_{\ell+1},\dots,z^0_{\nb}=\tilde{z}^0_{\nb}$. On pose $c:[0,1] \to (\RP^1)^{\nb} \setminus Diag_{\nb}$ tel que $c(0)=\underline{z}^0$ et $c(1)=\underline{\tilde{z}}^0$ un tel chemin. (Comme précédemment, on note parfois $c$ pour désigner $c([0,1])$ sans ambiguïté possible.) Puisque $u_0$ est une immersion, on rappelle que $\Nud=\OP(\nb-1)$ (propriété \ref{degnormal}), de plus $f^0_i \in H^0_{c_X}(\CP^1,\OP(-\widehat{z^0_i}))$ car $\mathcal{B}_0$ est une base modèle. On construit ainsi une homotopie de base $h_{t \in [0,1]}^0$ dans l'espace vectoriel $H^0_{c_X}(\CP^1,\OP(\nb-1))$ de sorte que pour $i \in \{1,\dots,\nb\}$, $h_0^{0}(f^0_i)=f^0_i$ et $h^0_1(f^0_i) \in H^0_{c_X}(\CP^1,\OP(-\widehat{\tilde{z}^0_i}))$ suivant l'ordre cyclique des zéros de sections. D'après le choix de la base $(\tilde{e}_1^0,\dots,\tilde{e}_{k_d}^0,\tilde{f}_1^0,\dots,\tilde{f}_{\nb}^0)$, les vecteurs $\tilde{e}_i$ de $T_{\tilde{z}_i}\RP^1$ sont tous orientés positivement relativement à $\mathfrak{o}_{\R A_0}$. Puisque $\Rev(c) \subset D_{\star}$ on a $\Rev(c) \cap \Omega_{\gamma} = \emptyset$ et on obtient une famille de sections $(\tilde{f}_1^0,\dots,\tilde{f}_{k_d}^0)$ qui est ordonnée avec $(\tilde{e}_1^0,\dots,\tilde{e}_{k_d}^0)$  dans $T_{c(1)}\RM$ car $h^0_1(f^0_i) \in H^0_{c_X}(\CP^1,\Nuidtd)$ et positive car $\angle(\tilde{e}_i^0,h^0_1(f^0_i))= \angle(e_i^0,f^0_i)= +1$, pour $i \in \{1,\dots,\nb\}$. On en déduit l'existence de réels $\lambda^i_+ > 0$ tels  que $h^0_1(f^0_i)=\lambda^i_+\tilde{f}^0_i$ pour $i \in \{1,\dots,\nb\}$. En conclusion, $(f^0_1,\dots,f_{\nb}^0)$ et $(\tilde{f}_1^0,\dots,\tilde{f}_{\nb}^0)$ sont homotopes comme bases de $H^0_{c_X}(\CP^1,\Nud)$ et le résultat s'en déduit par positivité de la base $\mathcal{B}_0$.
\end{proof}

\paragraph{Trivialisation le long du chemin.}\label{trivial}

On décrit une trivialisation du fibré tangent $\TMc$ le long de $\tilde{\gamma}$. À partir de celle-ci on va construire  une trivialisation le long de ${\gamma}$. Le chemin $\tilde{\gamma}$ est inclus dans le lieu régulier $\Reg$ du morphisme d'évaluation. En particulier le morphisme d'évaluation réel est un difféomorphisme et on peut relever la décomposition \eqref{decomp3} le long de $\tilde{\gamma}$
$$T|_{\tilde{\gamma}}\RMc=\bigoplus_{i=1}^{\nb}(\Rev)^{*}T|_{\tilde{\gamma}}\RX_{(i)}.$$
On pose $\tilde{\Phi}$ la trivialisation de $T|_{\tilde{\gamma}}\RMc$ issue de $\mathcal{B}_0$ et compatible avec cette décomposition et la suite exacte ${0 \rightarrow \Ker d\ob_i \rightarrow (\ev)^*T_{X_{(i)}} \stackrel{d\ob_i|}{\rightarrow} d\ob_i(\Ker d\hat{e}_{i}) \rightarrow 0}$ (voir proposition \ref{decompfib}) dont les faisceaux adjacents sont les faisceaux de fibrés en droite restreints à $\tilde{\gamma}$. On note $\tilde{e}_i:\tilde{\gamma} \to \Ker d\ob_i$ les $\nb$ sections tautologiques réelles de chaque sous-fibré $\Ker d\ob_i|_{\tilde{\gamma}}$ de sorte que  $\tilde{e}_i(\tilde{\gamma}(0))=\tilde{e}_i^0$ et $\tilde{f}_i:\tilde{\gamma} \to \Ker d\hat{e}_i/\Ker d\ob_i$ les $\nb$ sections tautologiques réelles telles que  $\tilde{f}_i(\tilde{\gamma}(0))=\tilde{f}_i^0$ pour $i \in \{1,\dots,\nb\}$. C'est-à-dire, en notant $\tilde{\mathcal{B}}_t=\basttd$ pour $t \in [0,1]$ on définit
\begin{align*}
\tilde{e}_i^t=\tilde{e}_i(\tilde{\gamma}(t))=\tilde{\Phi}(t;0,\dots,1,\dots,0,0,\dots,0)\\
\tilde{f}_i^t=\tilde{f}_i(\tilde{\gamma}(t))=\tilde{\Phi}(t;0,\dots,0,0,\dots,1,\dots,0)
\end{align*}
la section de bases de $\TRMc$ issues de la trivialisation $\tilde{\Phi}$ le long de $\tilde{\gamma}$. \`A présent, on définit une section de bases de $\TRMc$ le long du chemin $\gamma$. On considère pour cela une trivialisation $\phi$\label{phi} de $\Ker d\ob$ le long de $\gamma$, issue de $\mathcal{B}_0|_{\ker d\ob}$ et compatible avec la décomposition en droites du lemme \ref{propkero}. On note $e_i:\gamma \to \Ker d\ob_i$ les $\nb$ sections tautologiques de chaque sous-fibré $\Ker d\ob_i|_{\gamma}$ issues de $e^0_i \in \ker d_{\gamma(0)}\ob_i$. On considère ensuite $\tilde{f}_i:\tilde{\gamma} \to \TRMc/\Ker d\ob$ les $\nb$ sections précédemment définies de $T_{\RMc}/\Ker d\ob|_{\tilde{\gamma}} \cong \bigoplus_{i=1}^{\nb} \Ker d\hat{e}_i/\Ker d\ob_i|_{\tilde{\gamma}}$ afin d'obtenir une famille libre $\overline{\mathcal{B}}_t=\bast_{t \in [0,1]}$ de sections du fibré $T|_{\gamma}\RMc$ par le lemme suivant.

\begin{lemm}
Il existe un isomorphisme de fibrés
$$\TRMc/\Ker d\ob|_{\tilde{\gamma}} \cong \TRMc/\Ker d\ob|_{\gamma}$$
qui fait de $\{\tilde{f}^t_i;i=1,\dots,\nb\}_{t \in [0,1]}$ des sections de $\TRMc/\Ker d\ob|_{\gamma}$.
\end{lemm}

\begin{proof}
On a la suite exacte de fibrés $0 \rightarrow \Ker d\ob\rightarrow \TRMc \xrightarrow{d\ob} \TRMco \rightarrow 0$ d'où on tire les isomorphismes de fibrés $\TRMc/\Ker d\ob|_{\tilde{\gamma}} \cong T|_{\ob(\tilde{\gamma})}\RMco$ d'une part et $\TRMc/\Ker d\ob|_{\gamma} \cong T|_{\ob(\gamma)}\RMco$ d'autre part. Or les images des chemins $\ob(\tilde{\gamma})$ et $\ob(\gamma)$ dans $\RMco$ sont égales.
\begin{displaymath}
\xymatrix{[0,1] \ar[d]_-{\gamma} \ar[r]^-{\tilde{\gamma}} & \RMc|{\tilde{\gamma}} \ar[d]^-{\ob} \ar@/^/[r]^-{\tilde{f}_i} & \TRMc/\Ker d\ob|_{\tilde{\gamma}} \ar[l] \ar[dd]^-{\wr} \\
\RMc \ar[r]_-{\ob} & \RMco|_{\ob(\gamma)=\ob(\tilde{\gamma})} \\
\TRMc/\Ker d\ob|_{\gamma} \ar[u] \ar[rr]_-{\sim} & & T|_{\gamma}\RMco \ar[lu]}
\end{displaymath}
On obtient ainsi une famille de sections du fibré $\TRMc/\Ker d\ob|_{\gamma}$, notées abusivement $\{\tilde{f}^t_i;i=1,\dots,\nb\}_{t \in [0,1]}$. Chaque élément $\tilde{f}^t_i$, $i \in \{1,\dots,\nb\}$ est en définitive une section de $\Ker d\hat{e}_i/\Ker d\ob_i|_{\gamma}$.
\end{proof}

\begin{rema}
Les sections $e_i$ et $\tilde{e}_i$ sont identiques pour $i>\ell$, suivant l'identification $\ker d_{\gamma(t)}\ob_i=T_{z_i^t}\C_t=T_{\tilde{z}_i^t}\C_t=\ker d_{\tilde{\gamma}(t)}\ob_i$, quelque soit $t \in [0,1]$.
\end{rema}

\begin{rema}
Le chemin $\overline{\mathcal{B}} : [0,1] \rightarrow \prod^{2\nb}\TRMc$ définit bien une section de bases standards de $\TRMc$ le long de $\gamma$ puisque la famille $(e^t_1,\dots,e^t_{\nb})_{t \in [0,1]}$ (resp. $(\tilde{f}_1,\dots,\tilde{f}_{\nb})_{t \in [0,1]}$) est libre dans le sous-fibrés $\Ker d\ob|_{\gamma}$ (resp. dans le quotient $\TRMc/\Ker d\ob|_{\gamma}$).
\end{rema}

\paragraph{Retour au modèle.}\label{basa}
On décrit une homotopie de la base standard $\mathcal{\overline{B}}_1$ de $\TRMa$ sur une base modèle afin de l'évaluer par le morphisme d'évaluation réel. On utilise les techniques employées dans la démonstration du lemme \ref{turnlemma} au paragraphe \ref{turn}.

\begin{lemm}\label{c1}
Il existe une base modèle de $\TRMa$ notée $$\mathcal{B}_1=\basa$$ ayant les propriétés suivantes~:
\begin{enumerate}
\item $(\tilde{f}^1_1,\dots,\tilde{f}_{\nb}^1)$ et $(f^1_ 1,\dots,f^1_{\nb})$ sont homotopes dans $H_{c_X}^0(\CP^1,\Nua)$ ;
\item pour $i \in \{1,\dots,\ell\}$~: $f_i^1 \in H_{c_X}^0(\CP^1,\mathcal{N}_{u,-\ob_{\ell+1-i}(z)})$ ;
\item pour $i \in \{\ell+1,\dots,\nb\}$~: $f_i^1 \in H_{c_X}^0(\CP^1,\Nui)$.
\end{enumerate}
\end{lemm}

\begin {proof}
Quitte à restreindre $\gamma$ et d'après le lemme \ref{maprtilde}, l'ordre cyclique des points marqués à l'image $u_1(\underline{z} \cup \underline{\tilde{z}})$ sur $\RA_1$ où $A_1=u_1(\CP^1)$ est le suivant
\begin{eqnarray*}
\begin{aligned}u_1(z_1)&<\dots<u_1(z_{\frac{\ell}{2}})<u_1(\tilde{z}_{\ell})<\cdots<u_1(\tilde{z}_{\frac{\ell}{2}+1})<u_1(z_{\nb})<\cdots \\ \cdots& <u_1(z_{k'+1})<\dots <u_1(\tilde{z}_{\frac{\ell}{2}})<\dots<u_1(\tilde{z}_1)<u_1(z_{\frac{\ell}{2}+1})<\dots<u_1(z_{k'})\end{aligned}\end{eqnarray*}
\begin{figure}[htp]
\begin{center}
\input{config1.pstex_t}
\end{center}
\caption{$X=\CP^2$, $d=3.[L]$, $\nb=8$, $\ell=2$}\label{config1}
\end{figure}\\
On choisit des représentants $(\underline{z}^1,u_1)$ et $(\underline{\tilde{z}^1},u_1)$ dans $(\RP^1)^{\nb} × \RMor_d(X,c_X)$ de $\mapa$ et $\mapatd$ de sorte que $z^1_i=\tilde{z}^1_i$ lorsque $i \in \{\ell+1,\dots,\nb\}$. On rappelle que, $u_1$ étant une immersion, $\Nua=\OP(\nb-1)$ et de plus $\tilde{f}^1_i \in H^0_{c_X}(\CP^1,\OP(-\widehat{z^1_i}))$. On peut donc définir comme au paragraphe \ref{turn} une homotopie de base, notée $h^1_{t\in[0,1]}$, dans $H^0_{c_X}(\CP^1,\OP(\nb-1))$ de sorte que $h^1_0=id$, $h^1_1(\tilde{f}^1_i) \in H^0_{c_X}(\CP^1,\OP(\widehat{z^1_{\ell+1-i}}))$ lorsque $i \in \{1,\dots,\ell\}$ et $h^1_1(\tilde{f}^1_i) \in H^0(\CP^1,\OP(\widehat{z^1_{i}}))$  $i \in \{\ell+1,\dots,\nb\}$ d'après l'ordre cyclique sur les zéros de sections. En posant $f^1_i=h^1_1(\tilde{f}^1_i)$ pour tous $i \in \{1,\dots,\nb\}$, on obtient le résultat annoncé.
\end{proof}

La base $\mathcal{B}_1=\basa$ est modèle puisque pour chaque $i$ dans $\{1,\dots,\nb\}$, $f^1_i$ appartient à $H^0(\CP^1,\Nua)$, mais elle n'est pas ordonnée. On détermine son signe relativement à l'orientation $\mathfrak{o}_{\gamma}$ à la partie \ref{evalbase}.

\subsubsection{Trivialisation. Cas $d'=0$}\label{orient0}

On construit une trivialisation de $\TRMc$ au-dessus de $\gamma$. Le long du chemin, on trivialise la partie définie par les $\nb$ derniers vecteurs $(f^0_1,\dots,f^0_{\nb})$ de la base standard $\mathcal{B}_0$ en se donnant $\nb$ sections constantes $t \in [0,1] \mapsto f^0_i \in H^0(\CP^1,\Nud)$ puisque l'application $u_t=u_0$ est constante. Pour la partie de $T|_{\gamma}\RM$ engendrée par $\Ker d\ob$ on ne peut pas définir les vecteurs $e_i^t$ comme précédemment car $$\ker d|_{\mapr}\ob \cap \ker d|_{\mapr}\Rev \neq \{0\}.$$ Puisque l'application $u_0$ est constante le long du chemin, on peut décrire la courbe universelle au-dessus $\gamma$ comme l'espace des déformations de la source. Considérons le produit $[0,1] ×\CP^1 ×\{u_0\}$ et les $\nb$ sections $t \in [0,1] \mapsto z_i^t \in \CP^1$ définies par le chemin $\gamma$. La courbe universelle $\mathcal{U}^d_{\nb}(X)|_{\gamma}$ (voir théorème \ref{Mck}) est l'éclaté $\Bl_{\xi}([0,1] ×\CP^1)$ du produit $[0,1] \times \CP^1$ en l'unique point de concours $(t_{\star},\xi)$ des $k'$ premières sections. (Voir figure \ref{Ugamma}.)
\begin{figure}[htp]
\begin{center}
\input{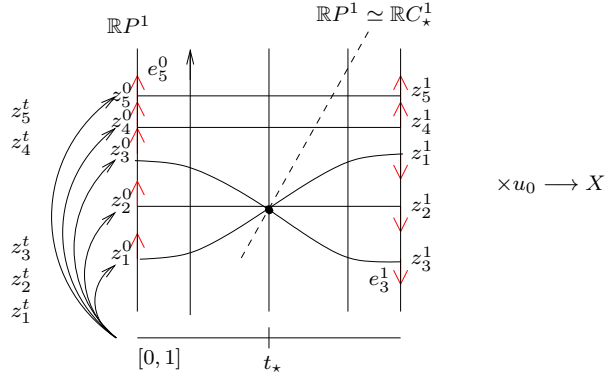}
\end{center}
\caption{$d'=0$, $u_t=u_0$, $\nb=5$, $k'=3$}\label{Ugamma}
\end{figure}\\
On définit ainsi une trivialisation de base $(e_1^t,\dots,e_{\nb}^t)_{t \in [0,1]} \subset \Ker d\ob$ dans $\Bl_{\xi}([0,1] \times \CP^1)$ issue de $(e_1^0,\dots,e_{\nb}^0)_{t=0}$.

\subsubsection{Évaluation de bases}\label{evalbase}

On compare les orientations induites par les bases  $\mathcal{B}_0$ et $\mathcal{B}_1$ dans une trivialisation du fibré tangent à l'image du chemin par le morphisme d'évaluation. Lorsque ces orientations coïncident, on en déduit que $\epsilon_{d',k'}=0$, dans le cas contraire on obtient $\epsilon_{d',k'}=1$.

\paragraph{Trivialisation à l'image. Cas $d'\neq0$.}\label{trivRX}

On se place dans une trivialisation de $T_{\RX^{\nb}}$ le long de $\gamma_*$ qui soit compatible avec le morphisme d'évaluation. On définit une trivialisation de $T|_{\gamma_*}\RX^{\nb}$ dont la base standard au-dessus de $\gamma_*(1)$ définit une base modèle ordonnée positive selon le procédé du paragraphe \ref{basd}. Reprenons la trivialisation $\phi$ définie en \ref{phi} ayant pour sections tautologiques $e_i : \gamma \subset \RMc \to \bigoplus_{i=1}^{\nb}\Ker d\ob_i|_{\gamma}$ pour $i=1,\dots,\nb$. Puisque $d\Rev$ est injective sur $\Ker d\ob|_{\gamma}$ (cf. proposition \ref{kerev}), on construit une famille libre de $\nb\ (=\dim \Ker d\ob)$ sections $v_i: \gamma_* \subset \RX^{\nb} \to T|_{\gamma_*}\RX^{\nb}$ comme images des sections $e_i^t$
$${v}_i(\gamma_*(t))=d\Rev(e_i^t)\; ;\ t \in [0,1].$$
Pour plus de clarté, on notera ${v}_i^{t}=v_i(\gamma_*(t))$. Quelque soit $t$ dans $[0,1]$ et $i$ dans $\{1,\dots,\nb\}$, $v_i^t$ appartient à $T_{\gamma_*(t)}\RX_{(i)}$. On complète la base en choisissant pour chaque $v_i^0$ un vecteur $w_i^0$ dans $T_{\gamma_*(0)}\RX_{(i)}$ de sorte que $\overline{w}_i^0=d\Rev(f_i^0)$ où $\overline{w}_i^0$ représente la classe au quotient de $w_i^0$ dans la suite exacte $$\xymatrix{0 \ar[r]& <v^0_i>_{\R} \ar[r]& T_{\gamma_*(0)}\RX_{(i)} \ar[r]& T_{\gamma_*(0)}{\RX_{(i)}}/<v^0_i>_{\R} \ar[r]& 0.}$$ Ainsi le couple $(v_i^0,w_i^0)$ est une base positive de $T_{\gamma_*(0)}\RX_{(i)}$ relativement à $\mathfrak{o}_{\gamma}$. On considère alors une trivialisation de $T|_{{\gamma}_*}\RX_{(i)}$ et on complète chaque base $(v^t_1,\dots,v^t_{\nb})_{t \in [0,1]}$ par une section issue de $w_i^0$, on obtient une famille libre de $\nb$ sections ${w}_i:\gamma_* \subset \RX^{\nb} \to T|_{\gamma_*}\RX^{\nb}$ telles que $\overline{w}_i^{t}=\overline{w}_i(\gamma_*(t)) \in \left(T_{\RX_{(i)}}/<v_i^t>_{\R}\right)_{t \in [0,1]}$ suivant le diagramme $$\xymatrix{0 \ar[r]& <v_i>_{\R}|_{\gamma_*} \ar[r] & T|_{\gamma_*}\RX_{(i)} \ar[d] \ar[r] & T_{\RX_{(i)}}/<v_i>_{\R}|_{\gamma_*} \ar[r] & 0 \\ & & \gamma_* \subset \RX_{(i)} \ar@/_/[u]_-{w_i} \ar@/_/[ru]_-{\overline{w}_i} & }$$
Ainsi, quelque soit $t$ dans $[0,1]$, $\mathcal{V}_t=\{{v}_1^t,\dots,{v}_{\nb}^t,{w}_1^t,\dots,{w}_{\nb}^t\}$ est une base positive de $T_{\gamma_*(t)}\RX^{\nb}$ relativement à $\mathfrak{o}_{\gamma}$ et compatible avec la décomposition qui suit (définie par le choix de $w_i^0$) $$T_{\gamma_*(t)}\RX^{\nb}=\bigoplus_{i=1}^{\nb}T_{\gamma_*(t)}\RX_{(i)} \cong \bigoplus_{i=1}^{\nb}\left(<v_i^t>_{\R} \oplus \left(T_{\gamma_*(t)}\RX_{(i)}/<v_i^t>_{\R}\right)\right).$$
\begin{rema}
Par construction $d|_{\gamma(1)}\Rev(\mathcal{B}_0)=\mathcal{V}_0$ et $(d\Rev)^{-1}(\mathcal{V}_1)$ est une base modèle ordonnée positive de $\TRMa$.
\end{rema}

\paragraph{Trivialisation à l'image. Cas $d'=0$.}\label{trivX0}

On considère l'unique trivialisation (à homotopie près) issue de $\Rev(\mathcal{B}_0)=(v^1_0,\dots,v_{\nb}^0,w_1^0,\dots,w_{\nb}^0)$ le long de $\gamma_*$ dont une base est donnée par un vecteur tangent à $\RA_0=\R A_t\, \forall t \in [0,1]$ à chaque point marqué $u_t(z_i)$ et son complémentaire \emph{positif} pour l'orientation définie par $\mathfrak{o}_{\gamma}$ dans $\RX_i$ pour tous $i \in \{1,\dots,\nb\}$. L'ensemble des $2(\nb-k')$ vecteurs $\{v_{k'+1}^t,\dots,v_{\nb}^t,w_{k'+1}^t,\dots,w_{\nb}^t\}$ sont choisis constants quelque soit $t$ dans $[0,1]$ puisque les sections $t \mapsto z_i^t$, $k'<i\leq \nb$ sont tautologiques ($t \mapsto z_i^0$) par définition de $\gamma$.
\begin{figure}[htp]
\begin{center}
\input{triv0.pstex_t}
\end{center}
\caption{$X=\CP^2$, $d=2[L]$, $k'=3$ }\label{triv0}
\end{figure}
Le résultat au point $\Rev \mapa$ est une base $\mathcal{V}_1$ de $T|_{\gamma_*(1)}\RX^{\nb}$ constituée de $\nb$ vecteurs $(v^1_1,\dots,v^1_{\nb})$ tangents à $\RA_1=\RA_0$ aux points $u_1(z_i^1)$ de sorte que l'orientation qu'ils induisent sur $\RA_1$ est toujours positive relativement au choix de $\mathfrak{o}_{\RA_0}$ ; puis de $\nb$ vecteurs $(w_1^0,\dots,w_{\nb}^0)$ complémentaires dans $T_{u_1(z_i^1)}\RX_{(i)}$ le tout formant une base positive pour le choix de $\mathfrak{o}_{\gamma}$. (Voir figure \ref{triv0}.)

\subsubsection{Matrices de transition}\label{mattrans}

On donne la matrice de l'image de la base $\mathcal{B}_1$ dans la base $\mathcal{V}_1$ par le morphisme d'évaluation réel. D'après le théorème \ref{W3}, $\tilde{\mathcal{B}}_1$ est une base modèle ordonnée positive de $\TRMatd$ puisque $\tilde{\gamma} \subset \Reg$ et $\tilde{\gamma}_* \cap \W = \emptyset$.
Plus précisément, le sens $\angle(\tilde{e}_i^1,\tilde{f}_i^1)$ est $+1$ relativement à $\mathfrak{o}_{\gamma}$ quelque soit $i$ dans $\{1,\dots,\nb\}$. On rappelle que le long de $\tilde{\gamma}$, chaque couple $(\tilde{e}_i^t,\tilde{f}^t_i)$ constitue une base du sous-espace $(d_{\tilde{\gamma}(t)}\Rev)^{-1}(T_{\tilde{\gamma}_*(t)}\RX_{(i)})$ de $T_{\tilde{\gamma}(t)}\RMc$.

\begin{lemm}\label{signarrivee}
Soit $\mathcal{B}_1=\basa$ la base modèle précédemment définie de $\TRMa$, on a~:
\begin{enumerate}
\item $\angle(e_i^1,f_i^1)=+1$ pour $\ell<i\leq\nb$~;
\item $\angle(e_{i}^1,f_{\ell+1-i}^1)=-1$ pour $i \in \{1,\dots,\ell\}$.
\end{enumerate}
\end{lemm}

\begin{proof}
On fixe une orientation $\mathfrak{o}_{\R A_1}$ sur la partie réelle de $A_1=u_1(\CP^1)$ de sorte que les vecteurs $\{e^1_i ; i=k'+1,\dots,\nb\}$ (c.-à-d. les éléments de la base dans $T_{C_{\star}^2}$) soient positifs. On rappelle que $d''\in H_2(X,\Z)$ est non nulle et réelle donc la partie réelle de l'image de $C_{\star}^2$ n'est ni vide ni réduite à un point de $\RX$ car elle contient des points marqués. Les vecteurs $\{e^1_i ; i=1,\dots,k'\}$ sont nécessairement orientés négativement relativement à $\mathfrak{o}_{\R A_1}$ puisque le chemin $\gamma$ est transverse au diviseur $\Kd$ de la frontière (cf. lemme \ref{mapr}). De même, en $\tilde{\gamma}(1)$ les vecteurs $\{\tilde{e}^1_i ; i=1,\dots,\ell\}$ et $\{\tilde{e}^1_i ; i=k'+1,\dots,\nb\}$ sont orientés positivement relativement à $\mathfrak{o}_{\R A_1}$ alors que les vecteurs $\{\tilde{e}^1_i ; i=\ell+1,\dots,k'\}$ le sont négativement (cf. lemme \ref{maprtilde}). Reprenons la construction de $\mathcal{B}_1$ et considérons le relevé du chemin $c'$ dans $(\RP^1)^{\nb} × \RMor_d(X,c_X)^*$ et l'homotopie de base $h^1$ de $H^0_{c_X}(\CP^1,\OP(\nb-1))$ associée (cf. lemme \ref{c1}). Une trivialisation de $T|_{c'}\RM$ issue de $\tilde{\mathcal{B}}_1=(\tilde{e}^1_1,\dots,\tilde{e}^1_{\nb},\tilde{f}^1_1,\dots,\tilde{f}^1_{\nb})$ compatible avec la suite exacte (\ref{decomp}) et définie par cette homotopie fournit une base de $\TRMa$ dont les $\nb$ derniers vecteurs sont $(f^1_1,\dots,f^1_{\nb})$. Puisque l'image $\Rev(c')$ n'intersecte pas $\W$ car elle est contenu dans $D_{\star}$ on obtient le long de la première composante une trivialisation $(\tilde{e}_1^1(t),\dots,\tilde{e}_{\nb}^1(t))$, $t \in [0,1]$ de $T|_{c'}(\RP^1)^{\nb}$ qui vérifie $\angle(\tilde{e}^1_i(1),f^1_i)=+1$ relativement à $\mathfrak{o}_{\gamma}$ car $\angle(\tilde{e}_i^1(0),\tilde{f}^1_i) = +1$, pour $i \in \{1,\dots,\nb\}$. Cependant, pour $i \in \{1,\dots,\ell\}$, $e^1_i$ étant orienté négativement relativement à $\mathfrak{o}_{\R A_1}$, il est opposé à l'orientation définie par $\tilde{e}_{i}^1(0)$. Or, cette orientation est aussi celle définie par $\tilde{e}^1_{\ell+1-i}(1) $ car $\ell+1-i \in \{1,\dots,\ell\}$. Alors que, pour $i \in \{\ell+1,\dots,\nb\}$, $e^1_i$ est orienté positivement relativement à $\mathfrak{o}_{\R A_1}$ et donc relativement à $\tilde{e}^1_i(1)$. Finalement, $\tilde{e}^1_i(1)= \lambda_+^i e_i^1$ où $\lambda_+^i \in \R_+^*$ si $i \in \{\ell+1,\dots,\nb\}$ alors que $\tilde{e}^1_{\ell+1-i}(1)=\lambda_-^i e_i^1$ avec $\lambda_-^i \in \R_-^*$ lorsque $i \in \{1,\dots,\ell\}$. En conséquence, $f^1_{\ell+1-i}(z^1_i)=\lambda_-^i \tilde{f}^1_i(\tilde{z}^1_i)$ pour $i \in \{1,\dots,\ell\}$ et $f^1_i(z^1_{i})=\lambda_+^i \tilde{f}^1_i(\tilde{z}^1_i)$ pour $i \in \{\ell+1,\dots,\nb\}$, avec $\lambda_-^i<0$ et $\lambda_+^i>0$, d'où le résultat.
\end{proof}

\begin{lemm}\label{matrat}
Soit $d' \in red^d$, $k' \in \{k_{d'}+2,\dots,\nb\}$, $\gamma$ un chemin défini comme en \ref{chemin} et $\mathcal{B}_1$, $\mathcal{V}_1$ les bases précédemment définies de $\TRMa$ et $T_{\gamma_*(1)}\RX^{\nb}$ respectivement. Alors, dans la trivialisation définie en \ref{trivRX}, l'image de la base $d_{\gamma(1)}\Rev(\mathcal{B}_1)$ dans la base $\mathcal{V}_1$ s'écrit matriciellement
\[\mathcal{B}_1|_{\mathcal{V}_1}=\left(\begin{array}{c|c|c}I_{\nb} & 0 & 0\\ \hline 0 & -J_{\ell}& 0 \\ \hline 0 & 0 & I_{\nb-\ell}\end{array}\right)\]
où $J=\left(\begin{array}{cccc}0 & \cdots & 0 & 1 \\ 0 & \cdots & 1& 0\\ \vdots & {\mathinner{\mkern2mu\raise1pt\hbox{.}\mkern2mu
\raise4pt\hbox{.}\mkern2mu\raise7pt\hbox{.}\mkern1mu}} & \vdots & \vdots \\ 1 & \cdots & 0 & 0\end{array}\right)$ est la matrice de permutation inversée.
\end{lemm}

\begin{proof}
La base de $T_{\gamma_*(1)}\RX^{\nb}$ définie par l'image de $\mathcal{B}_1$ est donnée par $d_{\gamma(1)}\Rev(e^1_1,\dots,e^1_{\nb})$ sur ses premières composantes, c'est-à-dire les premières composantes $(v_1^1,\dots,v_{\nb}^1)$ de $\mathcal{V}_1$ et dans le même ordre (cf. \ref{trivRX}). D'autre part, on a vu que lorsque $i \in \{\ell+1,\dots,\nb\}$, $f^1_i = \lambda_+^i (d\Rev)^{-1}(w_i^1)$ avec $\lambda_+^i > 0$ et donc, après éventuellement avoir normalisé, on obtient l'égalité $d_{\gamma(1)}\Rev(f^1_{\ell+1},\dots,f^1_{\nb})=(\overline{w}^1_{\ell+1},\dots,\overline{w}^1_{\nb})$. Enfin, pour $i \in \{1,\dots,\ell\}$, $f^1_i = \lambda^i_- (d\Rev)^{-1}(w_{\ell+1-i}^1)$ avec $\lambda^i_- <0$ et donc, après avoir éventuellement normalisé les bases, $d_{\gamma(1)}\Rev(f^1_{1},\dots,f^1_{\ell})=(-w_{\ell},\dots,-w_{1})$. La matrice associée à cette description est donc la matrice carrée $-J$ de taille $\ell$.
\end{proof}

Lorsque $d'=0$ on a construit une base standard non ordonnée $\mathcal{B}_1 = \basa$ de $\TRMd$ dont on décrit les propriétés relativement au morphisme d'évaluation dans le lemme suivant.

\begin{lemm}\label{matnul}
Soit $d'=0$, $k' \in \{2,\dots,\nb\}$, $\gamma$ un chemin défini comme en \ref{chemin} et $\mathcal{B}_1$, $\mathcal{V}_1$ les bases précédemment définies de $\TRMa$ et $T_{\gamma_*(1)}\RX^{\nb}$ respectivement. Alors, dans la trivialisation définie en \ref{trivX0} et après normalisation, la base image $d_{\gamma(1)}\Rev(\mathcal{B}_1)$ s'écrit matriciellement dans la base $\mathcal{V}_1$ \[\mathcal{B}_1|_{\mathcal{V}_1}=\left(\begin{array}{c|c|c|c}-I_{k'} & 0 & 0 & 0 \\ \hline 0 & I_{\nb-k'}& 0 & 0\\ \hline 0 & 0 & J_{k'}& 0 \\ \hline 0 & 0 & 0 & I_{\nb-k'}\end{array}\right).\]
\end{lemm}

\begin{proof}
Les sections concourantes $z_i^t$, $1 \leq i \leq k'$ induisent dans la trivialisation par $e_i^1$ une orientation opposée à celle fixée sur $\RC_0=\RC_1$ par $\mathfrak{o}_{\RA_0}$ puisque $\RC^1_{\star} = w_1(\R \mathcal{U}^d_{\nb}(X)|_{\gamma})^{\vee}$. Comme $\RA_0=\RA_1$, l'image de $e_i^1$ par $d\Rev$ est dans la direction de $v_i^1$ mais dans le sens opposé pour $1 \leq i \leq 0$. Les $\nb-k'$ dernières sections $z_i^t$, $k' < i$ ne rencontrant pas $w_1(\R \mathcal{\mathcal{U}}^d_{\nb}(X)|_{\gamma})$ (et pouvant être choisies tautologiques), l'image de $e_i^1$ par $d\Rev$ est dans la direction et le sens de $v_i^1$ dès que $k' < i$. (Voir figure \ref{Ugamma}). Finalement, d'après \ref{orient0},  $\mathcal{B}_1=(e_1^1,\dots,e_{\nb}^1,f_1^0,\dots,f_{\nb}^0)$, la matrice de permutation associée aux vecteurs de base dans la composante $H^0(\CP^1,\Nua)$ de la décomposition \eqref{decomp} est donnée par $J_{k'} \oplus I_{\nb-k'}$ et correspond au renversement des sections du faisceau normal issues de $(f^0_1,\dots,f^0_{k'})$ d'après le choix de $\gamma$. Chaque base $(v^1_i,d\Rev(f^1_{k'-i}))$ de $T_{\tilde{\gamma}(1)}\RX_i$ pour $1\leq i \leq k'$ étant positive puisque l'orientation de $\RA_1$ induite par $v^1_i$ est inchangée par rapport à l'orientation de $\RA_0$. De plus, par construction tous les autres vecteurs sont identiques après normalisation, c'est-à-dire que $d|_{\gamma(1)}\Rev(e^1_i)=v^1_i$ et $d|_{\gamma(1)}\Rev(f^1_i)=\overline{w}_i^1$ dès que $i>k'$.
\end{proof}

\subsubsection{Conclusion}

On évalue l'orientation induite par les bases d'arrivée des trivialisations précédemment définies. Puisque la base de départ était choisie positive et que l'image du chemin ne rencontre pas $\W$, il vient que cette orientation est compatible avec celle définie sur $\RX \setminus \W$ si et seulement si le déterminant des matrices précédemment définies est positif.

\begin{coro}
Soit $d' \in red^d$, $k'>k_{d'}$ et $k'>1$. Soit $\K$ une composante connexe de la frontière incluse dans $\RKd$ (comme décrit en \ref{diviseurs}) et de codimension un dans $\RMc$. Cette dernière participe à un représentant dual pour la première classe de Stiefel-Whitney de la composante connexe de $\RMc$ qui la contient si et seulement si  $E(\frac{k'-k_{d'}}{2}) \equiv 1 \mod (2)$ où $E$ désigne la partie entière.
\end{coro}

\begin{proof}
C'est une conséquence des lemmes précédents en observant que $$\det \left(\begin{array}{c|c|c}I_{\nb} & 0 & 0\\ \hline 0 & -J_{\ell}& 0 \\ \hline 0 & 0 & I_{\nb-\ell}\end{array}\right)=\det -J_{\ell}=(-1)^{\frac{\ell}{2}} \mod (2)$$ puisque $\ell \equiv 0 \mod (2)$ et $$\det \left(\begin{array}{c|c|c|c}-I_{k'} & 0 & 0 & 0 \\ \hline 0 & I_{\nb-k'}& 0 & 0\\ \hline 0 & 0 & J_{k}& 0 \\ \hline 0 & 0 & 0 & I_{\nb-k'}\end{array}\right)=(-1)^{k'} \det J_k=(-1)^{k'+E(\frac{k'}{2})}.$$
Il suffit de vérifier, d'une part que $\frac{\ell}{2} = E(\frac{k'-k_{d'}}{2})$ d'après le choix de $\ell$ et d'autre part que ${k'}+E(\frac{k'}{2}) \equiv E(\frac{k'+1}{2}) \mod (2)$. On rappelle que par convention $k_0=-1$ et on en déduit le résultat.
\end{proof}

\begin{coro}\label{corodet}
Soit $(X,c_X)$ une surface projective convexe équipée d'une structure réelle. La première classe de Stiefel-Whitney de $\RMc$ admet comme représentant
$$w_1(\RMc)=(\Rev)^*[w_1(\RX^{\nb})]+\sum_{\substack{d' \in \red\\k_{d'}<k'\leq \nb}}\epsilon_{d',k'}(\RKd)^{\vee}$$
avec $\epsilon_{d',k'} \in \{0,1\}$ et $\epsilon_{d',k'}=1$ si et seulement si $k'-k_{d'}=2 \textrm{ ou } 3 \mod (4)$.
\end{coro}

\begin{proof}
C'est la conséquence directe du corollaire précédent. On choisit l'expression qui découle de l'équivalence $$E(\frac{k'-k_{d'}}{2}) \mod (2) =1 \Leftrightarrow k'-k_{d'} = 2 \textrm{ ou } 3 \mod (4).$$
\end{proof}

\subsection{\'Etude du cas $\tau \neq 0$}

On s'intéresse aux autres structures réelles lorsque la permutaiton $\tau \in S_{\nb}$ est différente de l'identité. Les changements que cela impose à la démonstration précédente sont assez minimes et nous allons, autant que possible, nous appuyer sur ce les développements précédents. On dira d'un point de $X$ qu'il est \emph{imaginaire} (resp. \emph{réel}) s'il n'est pas (resp. s'il est) dans la partie réelle $\R X=\Fix (c_X)$. Pour le choix d'une composante $\K$ ($d' \in \red$, $k'>k_d'-1$), deux situations sont à envisager en fonction du nombre de paires de points marqués imaginaires par rapport au nombre $\ell$ (cf. définition \ref{ell}). Pour $\tau$ un élément d'ordre deux de $S_{\nb}$, on note $r_{\tau}$ le nombre d'éléments invariants (correspondant aux points marqués réels) et $s_{\tau}$ le nombre de permutations (correspondant aux paires imaginaires conjuguées) de sorte que $r_{\tau}+2s_{\tau}=\nb$. On commence par quelques préliminaires techniques.

\subsubsection{Bases et orientations}

Il existe deux structures réelles sur le produit $\CP^1\times\CP^1$. Nommément $c_1:(z_1,z_2)\mapsto(\overline{z}_1,\overline{z}_2)$ de partie réelle $\RP^1 \times \RP^1$ et $c_2:(z_1,z_2)\mapsto(\overline{z}_2,\overline{z}_1)$ de partie réelle homéomorphe à la sphère $S^2$. Considérer une permutation d'ordre deux $\tau$ revient à considérer dans le produit $(\CP^1)^k \setminus Diag_k$ de la construction de $\Mk$ une partie réelle homéomorphe au produit de $r_{\tau}$ droites projectives réelles et $s_{\tau}$ sphères. Les décompositions de l'espace tangent à $\R_{\tau}\Mc$ doivent donc différer de celles définies en \ref{secdecomptan} sur la partie \og sphérique \fg\ de $T_{\underline{z}}(\CP^1)^k$. Par contre, il n'y a pas de modification à apporter aux autres objets (on considère l'espace des section $H^0_{c_X^{\tau}}(C,\Nu)$ équivariantes pour $c_X^{\tau}$ etc.). Pour les mêmes raisons, l'espace tangent réel $T_{\underline{x}}{\R_{\tau}X^{\nb}}$ ne se décompose pas en produit direct de $T_{x_i}{\RX}$ mais est isomorphe au produit $(T_{*}{\RX})^r \oplus (T_{*}{\R_{\tau}X^2})^s$ où $\R_{\tau}X^2$ désigne la partie réelle de $c_{(X^2)}:(x_1,x_2)\mapsto(\overline{x}_2,\overline{x}_1)$.

\subsubsection{Cas $d'\neq0$}

\paragraph{Occurrence $r \geq \ell$}

Cette situation est déjà traitée dans la démonstration précédente. En effet, il suffit de considérer une indexation telle que les $\ell$ premiers points marqués soient tous réels. On reproduit ainsi en tout point la première étude (partie \ref{chemin}) en se dotant d'une base positive par la méthode déterminée dans la partie \ref{basd} dans une restriction aux points marqué réels. Les autres vecteurs étant choisis dans un relevé de $(T|_{\gamma_*}\R_{\tau}X^2)^s$ le long duquel la restriction $d|_{\tilde{\gamma}}\ev$ est injective (on ne \og déplace\fg\ pas les zéro de sections associés à ces points marqués).

\paragraph{Occurrence $r < \ell$}

On doit modifier le raisonnement seulement pour le sous-fibré $(T|_{\gamma_*}\R_{\tau}X^2)^s$ car on veut déplacer des zéros de sections de l'espace des déformations réelles à l'ordre un $H_{c_X^{\tau}}^0(C,\Nu)$ qui sont associés à des couples de points imaginaires conjugués. On pose $\varsigma=\min(\ell,2s)$ et on fait le choix d'une indexation pour laquelle les $\varsigma$ premiers points marqués sont des paires imaginaires conjuguées. Rappelons que le choix d'une orientation pour la courbe de départ $\RA_0$ et d'une orientation $\mathfrak{o}$ sur $\R X \setminus \W$ fournit une orientation sur $T_{\gamma(0)}\R\M^*$ l'espace tangent au point de départ du chemin. De façon équivalente, pour chaque couple de points marqués imaginaires conjugués, on définit une orientation sur la partie réelle homéomorphe à $S^2$ donnée par la structure complexe et qui ne dépend que du choix d'un des points du couple. Aussi, dans la procédure définie en \ref{basd} pour construire une base modèle positive au point de départ d'un chemin $\gamma$, on choisit pour chaque couple de points marqués $(z_i,z_j)$ une base réelle de l'espace tangent $T_{(z_i,z_j)}(\CP^1)^2$ déterminée par le choix de $z_i$ tel que $i<j$ (par convention). On obtient ainsi une base modèle positive de l'espace tangent au point $\gamma(0)$. On reproduit ensuite en tout point la démonstration précédente avec la structure réelle associée à la permutation $\tau$. Les homotopies successives de bases réelles n'ont d'autre effet, après avoir traversé la singularité, que d'inverser l'ordre des points marqués dans chaque paire imaginaires conjuguées. (Voir figure \ref{compconj}.)
\begin{figure}[htp]
\begin{center}
\input{compconj.pstex_t}
\caption{}\label{compconj}
\end{center}
\end{figure}\\
Donc l'orientation induite sur la partie réelle de la source $T_{(z_i,z_j)}(\CP^1)^2$ au point d'arrivée serait opposée à celle définie au point de départ $\gamma(0)$ si une telle transformation concernait les points marqués. On en déduit qu'il en est ainsi pour les homotopies qui affectent les zéros de sections équivariantes. On en conclut que les résultats précédents sont inchangés puisque chaque paire de points marqués imaginaires conjugués qui vont différer dans les présentations de $\underline{z}$ et $\underline{\tilde{z}}$ contribuent dans la matrice de transition (voir \ref{mattrans}) à une permutation des sections avec un signe opposé (c.-à-d. un élément de type $-J_2$). Autrement dit, dans le lemme \ref{matrat} on remplace la matrice $\mathcal{B}_1|_{\mathcal{V}_1}$ par la suivante $$\left(\begin{array}{c|c|c|c|c}I_{\nb} & 0 & 0 & 0 & 0\\ \hline 0 & -J_{\ell-\varsigma}& 0 & 0 & 0\\ \hline 0 & 0 & -J_2 & 0 & 0\\ \hline 0 & 0 & 0 & \ddots & 0\\ \hline 0 & 0 & 0 & 0 & I_{\nb-\ell}\end{array}\right)$$
où la sous-matrice $-J_2$ est répétée $\frac{\varsigma}{2}$ fois ($\varsigma$ est un nombre pair). Pour cette raison, le déterminant de cette matrice est du même signe que celle du lemme \ref{matrat} et dépend seulement de $\frac{\ell}{2} \mod (2)$.

\subsubsection{Cas $d'=0'$}

Lorsque deux paires de points marqués imaginaires conjugués se rencontrent le long du chemin $\gamma$ (transverse à la composante $\mathcal{D}^d_{0,k'}$ de $\R_{\tau}\mathcal{K}_{0,k'}^d$), elles vont se permuter deux à deux pour induire le long d'une trivialisation une orientation différente dans un voisinage du point d'arrivée  $\gamma(1)$. En effet, avec les notations précédentes, $u_{\star}(z^{\star}_i) \in \R X$ dès que $1 \leq i\leq k'$ car sinon $\codim \mathcal{D}^d_{0,k'}\geq 2$. On note $\varrho$ le nombre de points réels ($\varrho \leq r$) concernés par le choix de $\mathcal{D}^d_{0,k'}$ (c'est-à-dire qui se \og déplacent\fg\ le long du chemin $\gamma$). Alors, si $k'= \varrho$ on ne change rien (en choisissant une bonne indexation) et si $k'>\varrho$, on suit le même raisonnement qui aboutit à remplacer dans le lemme \ref{matnul} la matrice $\mathcal{B}_1|_{\mathcal{V}_1}$ par la suivante \[\left(\begin{array}{c|c|c|c|c|c}-I_{k'} & 0 & 0 & 0 & 0 & 0\\ \hline 0 & I_{\nb-k'}& 0 & 0 & 0 & 0\\ \hline 0 & 0 & J_{\varrho}& 0 & 0 & 0\\ \hline 0 & 0 & 0 & J_2 & 0 & 0 \\ \hline 0 & 0 & 0 & 0 & \ddots & 0 \\ \hline 0 & 0 & 0 & 0 & 0 & I_{\nb-k'}\end{array}\right)\] où la sous-matrice $J_2$ est répétée $\frac{k'-\varrho}{2}$ fois ($k'-\varrho$ est pair par définition de la structure complexe $c^{\tau}_{\mathcal{M}}$ et parce que $\mathcal{D}^d_{0,k'}$ est de codimension un). Le déterminant de cette matrice est du même signe que celle du lemme \ref{matnul} dépendant seulement de $k' \mod (2)$.\\

En conclusion on retrouve dans tous ces cas les propriétés sur le déterminant développées dans le corollaire \ref{corodet} et on en déduit le théorème \ref{theo}.

\bigskip
\begin{flushright}
Universidad Complutense Madrid\\
Departamento de \'Algebra\\
puignau@mat.ucm.es
\end{flushright}
\end{document}